\documentclass[12pt]{amsart}
\usepackage{amssymb, mathrsfs}

\evensidemargin\paperwidth
\advance\evensidemargin-\textwidth
\advance\oddsidemargin-1in
\evensidemargin\oddsidemargin

\topmargin\paperheight
\advance\topmargin-\textheight
\advance\topmargin-1in

\title
{Measure Equivalence Rigidity and Bi-exactness of Groups}

\author{Hiroki Sako}
\address{Department of Mathematical Sciences, University of Tokyo, Komaba, Tokyo, 153-8914, Japan }

\email{hiroki@ms.u-tokyo.ac.jp, hiroki@math.ucla.edu}


\newcommand{\C}{\mathbb{C}}

\newcommand{\Z}{\mathbb{Z}}

\newcommand{\R}{\mathcal{R}}
\newcommand{\A}{\mathcal{A}}
\newcommand{\E}{\mathfrak{E}}
\newcommand{\M}{\mathcal{M}}

\newcommand{\B}{\mathcal{B}}
\newcommand{\G}{\mathcal{G}}
\newcommand{\tr}{\mathop{\mathrm{tr}}\nolimits}
\newcommand{\supp}{\mathop{\mathrm{supp}}\nolimits}

\newcommand{\range}{\mathop{\mathrm{range}}\nolimits}

\newcommand{\Tr}{\mathop{\mathrm{Tr}}\nolimits}

\theoremstyle{plain}
\newtheorem{theorem}{Theorem}
\theoremstyle{plain}
\newtheorem{proposition}[theorem]{Proposition}
\theoremstyle{plain}
\newtheorem{lemma}[theorem]{Lemma}
\theoremstyle{plain}
\newtheorem{corollary}[theorem]{Corollary}
\theoremstyle{plain}
\newtheorem{definition}[theorem]{Definition}
\theoremstyle{plain}
\newtheorem{framework}[theorem]{Framework}

\theoremstyle{remark}
\newtheorem{remark}[theorem]{Remark}
\theoremstyle{remark}

\theoremstyle{remark}
\newtheorem{acknowledgment}{Acknowledgment}

\begin{document}

\begin{abstract}
We get three types of results on measurable group theory; direct product groups of Ozawa's class $\mathcal{S}$ groups, wreath product groups and amalgamated free products. We prove measure equivalence factorization results on direct product groups of Ozawa's class $\mathcal{S}$ groups. As consequences, Monod--Shalom type orbit equivalence rigidity theorems follow. We prove that if two wreath product groups $A \wr G$, $B \wr \Gamma$ of non-amenable exact direct product groups $G$, $\Gamma$ with amenable bases $A$, $B$ are measure equivalent, then $G$ and $\Gamma$ are measure equivalent. We get Bass--Serre rigidity results on amalgamated free products of non-amenable exact direct product groups.
\end{abstract}

\maketitle

\section{Introduction}
Measurable group theory is a discipline which deals with the question how much structure on countable groups is preserved through measure equivalence. The notion of measure equivalence was introduced by Gromov \cite{gromov} as a variant of quasi-isometry. The field recently has attracted much attention since small measure equivalence classes were found (Furman \cite{Furman; Higher Rank Lattice}, Kida \cite{Kida}). The following is the definition of measure equivalence and ME couplings
given by M. Gromov.

\begin{definition}[\cite{gromov}, 0.5.E.]\label{Introduction; Definition of ME}
Let $G$ and $\Gamma$ be countable groups.
We say that $G$ is \textbf{measure equivalent $($ME$)$} to $\Gamma$,
when there exist a standard measure space $(\Sigma, \nu)$, a measure preserving
action of $G \times \Gamma$ on $\Sigma$ and measurable subsets $X, Y \subset \Sigma$
with the following properties:
\begin{eqnarray*}
    \Sigma = \bigsqcup_{\gamma \in \Gamma} \gamma X = \bigsqcup_{g \in G} g Y,
    \quad
    \nu(X) < \infty, \quad \nu(Y) < \infty.
\end{eqnarray*}
Then we use the notation $G \sim_{\rm ME} \Gamma$.
The measure space $\Sigma$ equipped with the $G \times \Gamma$-action
is called an \textbf{ME coupling of} $G$ \textbf{with} $\Gamma$. If the $G \times \Gamma$-action is ergodic, then $\Sigma$ is said to be \textbf{ergodic}.
\end{definition}

The relation $\sim_\mathrm{ME}$ is an equivalence relation among countable groups. The equivalence relation sometimes forgets much structures on groups. For example, arbitrary two amenable countable groups are ME (by Ornstein--Weiss \cite{Ornstein--Weiss}, Connes--Feldman--Weiss \cite{Connes--Feldman--Weiss} and the correspondence between measure equivalence and weak orbit equivalence \cite{Furman; OE rigidity}). On the other hand, for some group $\Gamma$, the other group $G$ is forced to have some algebraic structure when $G$ and $\Gamma$ are ME. The latter phenomena are called ME rigidity.

Measurable groups theory is closely related to ergodic theory of measure preserving group actions. By Furman's observation \cite{Furman; OE rigidity}, if two group actions on standard probability space $X$ essentially have a common orbit, or more generally if they are stably orbit equivalent, we naturally get an associated measurable coupling. We get a cross-sectional links with variegated fields at this point (see Shalom's survey \cite{Shalom}). By Murray--von Neumann's group measure space construction \cite{Murray--vN; IV}, we introduce operator algebraic structures on orbit equivalence relations.

The purpose of this paper is to show ME and orbit equivalence rigidity results on three types of countable groups; direct product groups, wreath product groups and amalgamated free product groups.

\section{Main results}\label{Section; Main Results}

Our argument begin with a general principle, which can be used for the three cases. In the following subsections, we state the principle and explain main results on individual cases.

\subsection{Measurable Embedding of Subgroups}

When we consider that the ME coupling $\Sigma$ gives an identification of two groups $G$ and $\Gamma$, we may understand that the following defines locations of subgroups in $\Sigma$.

\begin{definition}\label{Introduction; MEm of subgroups}
Let $\Sigma$ be an ME coupling of $G$ with $\Gamma$ $($or measurable embedding defined in Definition $\ref{definition; measurable embedding})$.
We say that a subgroup $H \subset G$
\textbf{measurably embeds into} a subgroup $\Lambda \subset \Gamma$ \textbf{in} $\Sigma$,
if there exists a non-null measurable subset $\Omega \subset \Sigma$
which is invariant under the $H \times \Lambda$-action
so that the measure of a $\Lambda$ fundamental domain is finite. Then we use the notation
$H \preceq_{\Sigma} \Lambda$. The measurable subset $\Omega$ is called a \textbf{partial embedding of} $H$
\textbf{into} $\Lambda$. $($We remark that for every $\Lambda$-invariant measurable subset $\Omega^\prime$, there exists a $\Lambda$-fundamental domain.$)$
\end{definition}

We will make use of a strategy which was developed for group von Neumann algebras. In the book \cite{Brown--Ozawa; Approximation}, Brown and Ozawa introduced the notion of bi-exactness defined on a discrete group $\Gamma$ and its family of subgroups $\G$. The notion was characterized by topological amenability on a relative boundary. They showed the following criterion: If $\Gamma$ is bi-exact relative to $\G$, then for any von Neumann subalgebra $N \subset L\Gamma$ with non-amenable (non-injective) relative commutant, we have $N \preceq_{L\Gamma} L\Lambda$ for some $\Lambda \in \G$. Here, the symbol $\preceq_{L\Gamma}$ stands for the embedding of corners, which was defined by Popa (\cite{Popa; StrR I, Popa; Betti Numbers}).
Bi-exactness also gives a criterion for measurable embedding, which will be a key ingredient of the three kinds of results. In Section \ref{Section; Bi-exactness}, we will quickly
review its definition and basic properties.

\begin{theorem}
[Theorem \ref{Theorem; Non embeddability implies amenability}]
\label{Introduction; NonEm implies Ame}
Let $\Sigma$ be an ergodic ME coupling between $G$ and $\Gamma$.
Suppose that $\Gamma$ is bi-exact relative to $\G$. Let $H$ be a subgroup of $G$. If the centralizer $Z_G(H) = \{g \in G \ | \ g h = h g, h \in H\}$ is non-amenable, then there exists $\Lambda \in \G$ satisfying $H \preceq_\Sigma \Lambda$.
\end{theorem}

\subsection{Results on Direct Products}
We will show Monod--Shalom type theorems for class $\mathcal{S}$ groups (see Section \ref{Section; Bi-exactness} for the definition of $\mathcal{S}$).
In the paper \cite{Monod--Shalom}, Monod and Shalom proved ME and orbit equivalence rigidity theorems
on class $\mathcal{C}$ groups.
Both families of groups contains non-elementary word-hyperbolic groups. But there exist class $\mathcal{S}$ groups which have normal infinite amenable subgroups (Ozawa \cite{Ozawa; Kurosh, Ozawa; An Example}), while the class $\mathcal{C}$
does not contain such groups.

\begin{theorem}
[Theorem \ref{Theorem; Factorization of Product Groups}]
\label{Introduction; Factorization of Direct Product}
Let $\{G_i \ | \ 1 \le i \le m\}$ be a finite family of non-amenable groups and let $\{\Gamma_j \ | \ 1 \le j \le n\}$ be a finite family of $\mathcal{S}$ groups.
Denote
$G = \prod_i G_i$, $\Gamma = \prod_j \Gamma_j$ and $H_i = \prod_{k \neq i} G_k$.
Suppose $m \ge n$.
If $G \sim_{\mathrm{ME}} \Gamma$, then $m = n$ and
there exists $\sigma \in \mathfrak{S}_n$
satisfying $G_{\sigma(j)} \sim_\mathrm{ME} \Gamma_j\ (1 \le j \le n)$.
\end{theorem}

Ozawa and Popa \cite{Ozawa--Popa; Prime Factorization} got factorization results on type $\mathrm{II}_1$-factors. The above theorem can be understood as a measurable group theory version of the result. By the correspondence between measure equivalence and stable (weak) orbit equivalence given by Furman \cite{Furman; OE rigidity}, we also get orbit equivalence rigidity theorems. The most typical one is

\begin{theorem}
[Theorem \ref{Theorem; OE Str Rigidity for Product Group}]
\label{Introduction; Factorization of Direct Product in SOE}
Let $G, \Gamma$ be groups as above. Let $\alpha$ be a free ergodic measure preserving $($e.m.p.$)$ $G$-action
on a standard probability measure space $X$ and let $\beta$ be a free e.m.p. $\Gamma$-action
on a standard probability space $Y$.
Suppose that any $G_i$ has no non-trivial normal finite subgroup
and that any $\Gamma_j$ is ICC $($group with no finite conjugacy class $\neq \{1\}$$)$.

If the actions are stably orbit equivalent and the $H_i$-actions
$\alpha |_{H_i}$ on $X$ are ergodic,
then $m = n$ and there exist $\sigma \in \mathfrak{S}_n$ and
embeddings of groups $\phi_i \colon G_{\sigma(j)} \rightarrow \Gamma_j$ such that
the $\Gamma$-action $\beta$ is conjugate to the induced action $\mathrm{Ind}_G^\Gamma(\alpha, \prod \phi_i)$.
\end{theorem}

See Subsection \ref{Subsection; OE Str Rigidity Thms} for the definition of induced actions. In Section \ref{Section; Factorization of Product Groups}, we will get a result on symmetric groups $\mathrm{Out}(\R), \mathcal{F}(\R)$
of relations $\R$ and prove rigidity results on groups with an amenable direct product factor. By using Furman's technique \cite{Furman; Higher Rank Lattice},
we have the following. A suitable description for our cases has been written in Monod and Shalom's paper \cite{Monod--Shalom}.

\begin{theorem}
[Subsection \ref{Subsection; OE Super Rigidity Type Theorems}]
\label{Introduction; Superrigidity for Direct Product}
Let $\{\Gamma_j \ | \ 1 \le j \le n\}$ be a finite family of non-amenable ICC groups in the class $\mathcal{S}$. Denote $\Gamma = \prod_{j = 1}^n \Gamma_j$. Let $\beta$ be a free e.m.p.\ $\Gamma$-action
on a standard probability space $Y$. Suppose that the restrictions of $\beta$
on $\Lambda_j = \prod_{l \neq j} \Gamma_l$ are ergodic.
Let $G$ be an arbitrary group and let $\alpha$ be an arbitrary free e.m.p.\ $G$-action on a standard probability space $X$. Suppose that $\alpha$ does not have non-trivial recurrent subsets $($mild mixing condition$)$. If the actions $\alpha$ and $\beta$ are stably orbit equivalent, then these actions are virtually conjugate.
\end{theorem}

See Definition 1.8 in Monod and Shalom's paper \cite{Monod--Shalom} for the definition of the mild mixing condition.

\subsection{Results on Wreath Products}
The \textbf{wreath product} $A \wr G$ of a group $G$ with \textbf{base} group $A$
is the group obtained by the semidirect product group
$A \wr G = (\oplus_{g \in G} A^{(g)}) \rtimes G$, where $A^{(g)}$ are the copies of $A$ and $G$ act on the direct sum $\oplus_G A^{(g)}$ by the Bernoulli shift $h ((a_g)_g) = (a_{h^{-1}g})_{g}$.

\begin{theorem}
[Section \ref{Section; Wreath Products}]
\label{Introduction; ME between Wreath Product Groups}
Let $G, \Gamma$ be non-amenable exact groups and let $H, \Lambda$ be infinite exact groups.
Denote by $\widetilde{G}$, $\widetilde{\Gamma}$ wreath products
$\widetilde{G} = A \wr (G \times H)$,
$\widetilde{\Gamma} = B \wr  (\Gamma \times \Lambda)$
with amenable bases $A$, $B$.
The following hold true:
\begin{enumerate}
\item
If $\widetilde{G} \sim_\mathrm{ME} \widetilde{\Gamma}$,
then $G \times H \sim_\mathrm{ME} \Gamma \times \Lambda$.
For an ergodic ME coupling $\Sigma$ of $\widetilde{G}$ with $\widetilde{\Gamma}$, there exist $(G \times H) \times (\Gamma \times \Lambda)$-invariant measurable subsets $\Omega \subset \Sigma$ which gives an ME coupling of $G \times H$ with $\Gamma \times \Lambda$ and satisfies
$[ \widetilde{\Gamma} : \widetilde{G} ]_{\Sigma} = [\Gamma \times \Lambda : G \times H]_{\Omega}$;

\item
Let $\alpha$ be a free e.m.p.\ $\widetilde{G}$-action on a standard probability space $X$ and let $\beta$
be a free e.m.p. $\widetilde{\Gamma}$-action on a standard probability space $Y$. Suppose that the restrictions $\alpha |_{G \times H}$ and $\beta |_{\Gamma \times \Lambda}$ are ergodic. If $\alpha$ and $\beta$ are stably orbit equivalent, then
$\alpha |_{G \times H}$ and $\beta |_{\Gamma \times \Lambda}$ are stably orbit equivalent.
\end{enumerate}
\end{theorem}

Popa proved very powerful rigidity theorems on Bernoulli shift actions of $w$-rigid groups (von Neumann rigidity \cite{Popa; StrR I, Popa; StrR II}, cocycle rigidity \cite{Popa; SR for $w$-rigid Bernoulli shift}).
He also proved a cocycle super-rigidity theorem for Bernoulli shift actions of groups, which are
typically given by products of infinite groups and non-amenable groups
(\cite{popa; sr of malleable actions with spectral gap}).
In the papers, Popa developed the deformation/spectral gap argument, which has been used for several rigidity results on Bernoulli shift actions
(Ioana \cite{Ioana; Wreath}, Chifan and Ioana \cite{Chifan--Ioana; Bernoulli shift}) and
amalgamated free products (Chifan and Houdayer \cite{Chifan--Houdayer; Prime factors}). We note here that our paper was deeply influenced by the above results, although we will not use the technique.

\subsection{Results on Amalgamated Free Products}
We will also prove the following Bass--Serre rigidity theorem in measurable group theory. Theorem \ref{Introduction; Amalgamated Free Product} admits an amalgamation over a amenable subgroup, while restricting each factor to direct product of two non-amenable groups.

\begin{theorem}
[Theorem \ref{Theorem; B--S rigidity}]
\label{Introduction; Amalgamated Free Product}
Let $G_i\ (i = 0, 1)$ be a countable group which is given by a direct product
of two non-amenable exact groups.
Let $\Gamma_j\ (j = 0, 1)$ be also such direct product groups.
Denote by $G = G_0 \ast_A G_1, \Gamma = \Gamma_0 \ast_B \Gamma_1$ free products
with amalgamations by amenable subgroups $A \subset G_i, B \subset \Gamma_j$.
Under the convention $1 + 1 = 0$, the following hold true:
\begin{enumerate}
\item
If $G \sim_\mathrm{ME} \Gamma$,  then $G_0 \sim_\mathrm{ME} \Gamma_j$ and $G_1 \sim_\mathrm{ME} \Gamma_{j+1}$
for some $j \in \{0,1\}$.
\item
Let $\alpha$ be a free e.m.p.\ $G$-action on a standard probability space $X$ and let $\beta$
be a free e.m.p. $\Gamma$-action on a standard probability space $Y$. Suppose that
the restrictions $\alpha |_{G_i}$ and $\beta |_{\Gamma_j}$ are ergodic.
If $\alpha$ and $\beta$ are stably orbit equivalent, then
there exists $j \in \{0, 1\}$ so that
$\alpha |_{G_i}$ and $\beta |_{\Gamma_{i + j}}$ are stably orbit equivalent for each $i \in \{0, 1\}$.
\end{enumerate}
\end{theorem}

We will also prove other results in Theorem \ref{Theorem; B--S rigidity},
which are analogous to the results shown by Alvarez and Gaboriau
\cite{Alvarez--Gaboriau; Free Products}. They proved measure equivalence and stably orbit equivalence
results on free products of measurably freely indecomposable $(\mathcal{MFI})$ groups. The class $\mathcal{MFI}$ is a quite
large class including groups whose first $\ell^2$-Betti numbers are $0$.

In \cite{Chifan--Houdayer; Prime factors}, Chifan and Houdayer proved a von Neumann algebraic rigidity
theorem for group measure space constructions $L^\infty X \rtimes \Gamma$ of free e.m.p.\ actions,
where group $\Gamma$ was required to be
a free product of direct product groups between infinite groups and non-amenable groups.
The assertion was much stronger than rigidity on orbit equivalence relations.
Prior to these results, in \cite{IPP; AFP}, Ioana, Peterson and Popa got Bass--Serre rigidity results on
von Neumann algebras and orbit equivalence relations given by free product groups of $w$-rigid groups.

\section{The Notion of Measure Equivalence and Measurable Embedding}
\subsection{Measurable Embedding}
The following notion will be useful throughout this paper, even if one is only interested in measure equivalence. This is a generalization of Gromov's measure equivalence.

\begin{definition}\label{definition; measurable embedding}
Let $G$ and $\Gamma$ be countable groups. $($We admit the case that they are finite$)$.
We say that the group $G$ \textbf{measurably embeds} into $\Gamma$,
if there exist a standard measure space $(\Sigma, \nu)$, a measure preserving
action of $G \times \Gamma$ on $\Sigma$ and measurable subsets $X, Y \subset \Sigma$
with the following properties:
\begin{eqnarray*}
    \Sigma = \bigsqcup_{\gamma \in \Gamma} \gamma X
    = \bigsqcup_{g \in G} g Y, \quad
    \mu(X) < \infty.
\end{eqnarray*}
Then we use the notation $G \preceq_{\rm ME} \Gamma$.
The measure space $\Sigma$ equipped with the $G \times \Gamma$-action
is called a \textbf{measurable embedding} of $G$ into $\Gamma$.
The measurable embedding $\Sigma$ is said to be \textbf{ergodic},
if the $G \times \Gamma$-action is ergodic.
\end{definition}

If the measure of the $G$ fundamental domain $Y$ is also finite, then the measure space $\Sigma$ gives an ME coupling between $G$ and $\Gamma$. As in the case of ME couplings (Lemma 2.2 in Furman \cite{Furman; Higher Rank Lattice}), if we have a measurable embedding of $G$ into $\Gamma$, there is ergodic one by using ergodic decomposition.

\begin{definition}\label{Definition; MEm constant}
For a measurable embedding $(\Sigma, \nu)$ of $G$ into $\Gamma$,
the following quantity is called
the \textbf{coupling index} of $\Sigma$ and denoted by $[\Gamma : G] _\Sigma$:
\begin{eqnarray*}
    [\Gamma : G]_{\Sigma} = \nu(Y) / \nu(X) \in (0, \infty],
\end{eqnarray*}
where $X$ is a $\Gamma$ fundamental domain and $Y$ is a $G$ fundamental domain. This definition does not depend on the choice of $X$ and $Y$.
\end{definition}

\begin{remark}\label{Remark; MEm}
\begin{enumerate}
\item
    The relation $\preceq_\mathrm{ME}$ is transitive;
    if $H \preceq_\mathrm{ME} \Lambda$ and $\Lambda \preceq_\mathrm{ME} \Gamma$,
    then $H \preceq_\mathrm{ME} \Gamma$. The
    proof is the same as that of $``\sim_\mathrm{ME}$'' $($\cite{Furman; Higher Rank Lattice}$)$.
\item
    If countable groups $G$ and $\Gamma$ satisfy $G \preceq_\mathrm{ME} \Gamma$ and if $\Gamma$
    is amenable $($resp.~exact$)$, then $G$ is also amenable $($resp.~exact$)$.
    The class $\mathcal{S}$ on countable groups has the same property
    $($see Sako \cite{Sako}$)$.
\item
    For a subgroup $\Lambda \subset \Gamma$,
    we can regard $\Gamma$ as a measurable embedding of $\Lambda$ into $\Gamma$,
    letting $\Gamma$ act from the right and
    $\Lambda$ act from the left. Then the coupling index $[\Gamma : \Lambda]_\Gamma$
    coincides with the index of the group inclusion.
\item
    Let $G, H \subset \Gamma$ be subgroups. We regard $\Gamma$ as the standard self coupling of $\Gamma$, on which $\Gamma \times \Gamma$ acts by the left-and-right translation. The groups satisfy $G \preceq_\Gamma H$ if and only if there exists
    $\gamma \in \Gamma$ such that $G \gamma H$ is a finite union of left $H$-cosets.
    This is equivalent to $[G : G \cap \gamma H \gamma^{-1}] < \infty$.
\end{enumerate}
\end{remark}

We introduce supports of partial embeddings.
\begin{definition}\label{Definition; Supports}
Let $H \subset G$, $\Lambda \subset \Gamma$ be subgroups.
Let $\Sigma$ be a measurable embedding of $G$ into $\Gamma$.
Choose a $\Gamma$ fundamental domain $X$ and a $G$ fundamental domain $Y$.
We define $\supp^\Gamma_X(H \preceq_{\Sigma} \Lambda) \in L^\infty X$
by the projection which corresponds to
\begin{eqnarray*}
\bigvee \{\gamma \chi(\Omega) \ |
\ \gamma \in \Gamma, \Omega \subset \Sigma
\textrm{\ gives\ } H \preceq_\Sigma \Lambda \} \in (L^\infty \Sigma)^\Gamma.
\end{eqnarray*}
We define $\supp^G_Y(H \preceq_{\Sigma} \Lambda) \in L^\infty Y$ by the projection which corresponds to
\begin{eqnarray*}
\bigvee \{g \chi(\Omega) \ | \ g \in G, \Omega \subset \Sigma
\textrm{\ gives\ } H \preceq_\Sigma \Lambda \} \in (L^\infty \Sigma)^G.
\end{eqnarray*}
We call them \textbf{$\Gamma$-support} and \textbf{$G$-support} of $H \preceq_{\Sigma} \Lambda$ respectively.
\end{definition}
We note that $H \preceq_\Sigma \Lambda$ in $\Sigma$ if and only if $p \neq 0$
(or $q \neq 0$).

\subsection{Stable Orbit Equivalence}
For a free measure preserving $G$-action $\alpha$ on a standard measure space $X$, we write
the equivalence relation of the action as
\begin{eqnarray*}
    \R_\alpha = \{ (g x, x) \ | \ x \in X, g \in G \}
    \subset X \times X.
\end{eqnarray*}
This gives an equivalence relation on $X$ with countable equivalence classes.
On the set $\R_\alpha$, we introduce a structure as a measurable set
by the identification $\R_\alpha \ni (g x, x) \mapsto (g, x) \in G \times X$. The measure on $\R_\alpha$ is the same as one defined in Feldman--Moore \cite{Feldman--Moore; II}.
In the case that $X$ is a finite measure space and that the $G$-action on $X$ is ergodic,
we consider the amplification of $\R_\alpha^s$ for $s \in (0, \infty]$.
We deal with stable orbit equivalence $($SOE$)$ between two group actions on standard measure space. We refer to Vaes' survey \cite{Vaes} and Furman's paper \cite{Furman; OE rigidity} with terminology \textit{weak orbit equivalence}. As in the case of ME coupling, we have the following.

\begin{lemma}\label{Lemma; MEm and SOE}
There exists an ergodic measurable embedding of $G$ into $\Gamma$
with coupling index $s \in (0, \infty]$, if and only if there exist a free e.m.p.\ $G$-action on a standard probability space $X$ and a free e.m.p.\ $\Gamma$-action
on a standard measure space $Y$ so that they are SOE with compression constant $s$, namely,
$\R_\alpha^s \cong \R_\beta$.
\end{lemma}

For the case of ME coupling, we are done in Lemma 3.2 in Furman \cite{Furman; OE rigidity}
and Remark 2.14 in Monod--Shalom \cite{Monod--Shalom}, but we explain the both cases.

\begin{proof}
Suppose that there exist a free e.m.p.\ $G$-action $\alpha$ on $X$ and
a free e.m.p.\ $\Gamma$-action $\beta$ on $Y$ which are SOE
with compression constant $s \in [1, \infty]$. We identify
measure space $X$ with a measurable subset of $Y$ and
the relation $\R_\alpha$ with $\R_\beta \cap (X \times X)$.
Then we can naturally regard the rectangular part
$\R_\beta \cap (X \times Y)$
as an ergodic measurable embedding of $G$ into $\Gamma$, by letting $G$ act on the first entries and $\Gamma$ act on the second entry. In turn, suppose $s < 1$. By replacing the roles on $G$ and $\Gamma$,
we get an ME coupling between $G$ and $\Gamma$ given by a rectangular part
of $\R_\alpha$.

An ergodic measurable embedding $\Sigma$
can be regarded as an measurable embedding given by
SOE, when the natural $G$-action on $X \cong \Lambda \backslash \Sigma$
is (essentially) free. If $G \preceq_\mathrm{ME} \Gamma$, we can always find such a measurable embedding $\Sigma$ by the following procedure. We take a standard probability space $(X_1, \mu)$
which is equipped
with a measure preserving, weakly mixing and free $G$-action.
Let $\Gamma$ act on $X_1$ trivially. We regard
$\Sigma^\prime = \Sigma \times X_1$ as
an ergodic measurable embedding, on which $G$ and $\Gamma$ act by the diagonal actions.
Since the $G$-action on the set
$\Gamma \backslash \Sigma^\prime = \Gamma \backslash \Sigma \times X_1$ is free,
we get stable (weak) orbit equivalence.
\end{proof}

\subsection{Function Valued Measures}
\label{Subsection; Function Valued Measure}
Let $(\Sigma, \nu)$ be a standard measure space equip\-ped with
a measure preserving free action of
a countable group $\Gamma$. Assume that the $\Gamma$-action
has a fundamental domain $X$.
For a subgroup $\Lambda \subset \Gamma$, there exists
a fundamental domain $X_\Lambda$ for the
$\Lambda$-action (for instance $X_\Lambda = \bigsqcup_{i \in I} \gamma_i X$,
where $\{\gamma_i\}_{i \in I}$ are representatives of the right cosets $\Lambda \backslash \Gamma$).

We denote by $\Tr$ the integration of elements in $L^\infty \Sigma$ given by the measure $\nu$.
We naturally define the $\Gamma$-action on the function space $L^\infty \Sigma$.
We define an application $\Tr_\Lambda$ on $\Lambda$-invariant positive functions
$(L^\infty \Sigma)^\Lambda_{+}$ by the integration on $X_\Lambda$, that is,
$\Tr_\Lambda (\phi) = \Tr(\chi(X_\Lambda) \phi ) \in [0, \infty]$. This definition does not depend on the choice of $X_\Lambda$.
For a $\Lambda$-invariant measurable set $\Omega \subset \Gamma$,
we write $\Tr_\Lambda (\Omega) = \Tr_\Lambda (\chi(\Omega))$.

Consider the natural inclusion
$L^\infty X \ni f \mapsto \iota(f) \in (L^\infty \Sigma)^\Lambda,
$
defined as
\begin{eqnarray*}
    \iota(f)(\gamma x) = f(x), \quad x \in X, \gamma \in \Gamma.
\end{eqnarray*}
We denote by $\E ^\Lambda_X$
the pull back of the preduals:
\begin{eqnarray*}
    \iota^* = \mathfrak{E}^{\Lambda}_X \colon L^1 ((L^\infty \Sigma)^\Lambda,
    \Tr_\Lambda) \longrightarrow L^1 (X).
\end{eqnarray*}
The space $L^1 ((L^\infty \Sigma)^\Lambda, \Tr_\Lambda)$ can be identified with the
space of the measurable $\Lambda$-invariant functions which are integrable on
$X_\Lambda$.

Let $\widehat{(L^\infty \Sigma)^\Lambda_+}$ and $\widehat{L^\infty X_+}$ be the extended positive cones. The former set consists of
the $[0, \infty]$-valued $\Lambda$-invariant measurable functions on $\Sigma$,
and the latter set consists of the $[0, \infty]$-valued measurable functions on $X$.
The completely additive extension $\E^\Lambda_X$ of $\iota^*$ is unique.
We call the extension $\E^\Lambda_X$ the \textbf{function valued measure} on
$\Lambda \backslash \Sigma$.
By choosing the fundamental domain as $X_\Lambda =
\bigsqcup_{i \in I} \gamma_i X$,
the function valued measure is written by
\begin{eqnarray*}
    \E^\Lambda_X (\phi) (x) = \sum_{i \in I} \phi (\gamma_i x),
    \quad \phi \in \widehat{(L^\infty \Sigma)^\Lambda_+},
\end{eqnarray*}
because every positive measurable function on $X_\Lambda$ can be written as
a countable sum of integrable functions and the equation holds true for all
integrable functions. It turns out that for
any $\Lambda$-invariant measurable subset $\Omega \subset \Sigma$,
the function $\E^\Lambda_X (\chi(\Omega))$ is a $(\{0,1,\cdots \} \sqcup \{\infty\})$-valued function.
For a $\Lambda$-invariant measurable set $\Omega \subset \Sigma$,
we also write $\E^\Lambda_X (\Omega) = \E^\Lambda_X (\chi(\Omega))$.

We get the following basic properties of function valued measures.
\begin{lemma}\label{Lemma; Function Valued Measure}
The function valued measure satisfies the following:
\begin{enumerate}
\item
    For $\phi \in \widehat{(L^\infty \Sigma)^\Lambda_+}$,
    we get
    \begin{eqnarray*}
        \Tr_\Lambda(\iota(f)\phi) = \int_X f \E^\Lambda_X(\phi) d\nu,
        \quad f \in L^\infty X.
    \end{eqnarray*}
    This condition determines $\E^\Lambda_X(\phi)$.
\item
    Let $\theta$ be a measure preserving transformation on $\Sigma$ commuting with the $\Gamma$-action.
    Denote by $\alpha$ a transformation on $X \cong \Sigma / \Gamma$ given by $\theta$.
    We get
    \begin{eqnarray*}
        \alpha(\E^\Lambda_X(\phi)) = \E^\Lambda_X(\theta(\phi)),
        \quad \phi \in \widehat{(L^\infty \Sigma)^\Lambda_+}.
    \end{eqnarray*}
\item
    For a measurable subset $W \subset X$,
    we get
    \begin{eqnarray*}
        \chi(W) \E^\Lambda_X(\phi) = \E^\Lambda_X(\chi(\Gamma W) \phi).
        \quad \phi \in \widehat{(L^\infty \Sigma)^\Lambda_+}.
    \end{eqnarray*}
\end{enumerate}
\end{lemma}

\begin{proof}
When $\phi \in \widehat{(L^\infty \Sigma)^\Lambda_+}$ is integrable on $X_\Lambda$, the first condition is the definition of $\E^\Lambda_X(\phi)$.
By the complete additivity of $\E^\Lambda_X$, the first assertion holds for a general $\phi$.

For the second assertion, we note that $\theta(X_\Lambda)$ is also a fundamental domain for
the $\Lambda$-action on $\Sigma$.
For $\phi \in \widehat{(L^\infty \Sigma)^\Lambda_+}$ and $f \in L^\infty X$, we have
\begin{eqnarray*}
      \Tr_\Lambda (\iota(f) \theta(\phi))
    = \int_{\theta(X_\Lambda)} \iota(f) \theta(\phi) d \nu
    = \int_{X_\Lambda} \theta^{-1} (\iota(f)) \phi d \nu
    = \Tr_\Lambda (\iota(\alpha^{-1}(f)) \phi).
\end{eqnarray*}
Since $\alpha$ is measure preserving, we get
\begin{eqnarray*}
      \Tr_\Lambda (\iota(\alpha^{-1}(f)) \phi)
    = \int_X \alpha^{-1}(f) \E^\Lambda_X(\phi) d\nu
    = \int_X f \alpha(\E^\Lambda_X(\phi)) d\nu.
\end{eqnarray*}
By the first assertion, we conclude $\alpha(\E^\Lambda_X(\phi)) = \E^\Lambda_X(\theta(\phi))$.

For a measurable subset $W \subset X$, we also get
\begin{eqnarray*}
   & &  \Tr_\Lambda (\iota(f) \chi(\Gamma W) \phi)
    =   \Tr_\Lambda (\iota(f \chi(W)) \phi)
    =   \int_X                f \chi(W) \E^\Lambda_X(\phi)  d \nu
\end{eqnarray*}
By the first assertion, we get the third assertion.
\end{proof}

\begin{lemma}\label{Lemma; FVM and MEm}
Let $H \subset G$ and $\Lambda \subset \Gamma$ be subgroups.
Let $\Sigma$ be a measurable embedding of $G$ into $\Gamma$. Choose a $\Gamma$ fundamental domain
$X \subset \Sigma$.
Then $H \preceq_\Sigma \Lambda$ if and only if
there exists an $H \times \Lambda$-invariant measurable subset $\Omega \subset \Sigma$
so that the essential range of $\E^\Lambda_X (\Omega)$ satisfies
$\range(\E^\Lambda_X (\Omega)) \not\subset \{0, \infty\}$.
\end{lemma}

\begin{proof}
If there exists a partial embedding $\Omega$ for $H \preceq_\Sigma \Lambda$, then
the function $\E^\Lambda_X (\Omega)$ is non-zero, non-negative and integrable.
Thus the essential range of the function intersects with positive integers.

Suppose that there exists an $H \times \Lambda$-invariant measurable subset $\Omega$ with the above
property. Denote $F = \E^\Lambda_X (\Omega)$.
Then there exists a positive integer $n$ such that the preimage $F^{-1}([1,n]) = W \subset X$
is non-null. Since the function $\chi(\Omega)$ is $H$-invariant, the function $F$ on $X$ is $H$-invariant
under the dot action $H \curvearrowright X \cong \Gamma \backslash \Sigma$. Thus the measurable subset
$W \subset X$ is $H$-invariant under the dot action,
and the measurable subset $\Omega^\prime = \Omega \cap \Gamma W$ is
$H \times \Lambda$-invariant. By Lemma \ref{Lemma; Function Valued Measure}, we get
\begin{eqnarray*}
    0 < \Tr_\Lambda (\Omega^\prime)
      = \int_X \E^\Lambda_X (\Omega \cap \Gamma W) d \nu
      = \int_X F \chi(W) d \nu < \infty.
\end{eqnarray*}
For a $\Lambda$ fundamental domain $X_\Lambda$ for $\Sigma$, the measurable set $\Omega^\prime \cap X_\Lambda$ is a $\Lambda$ fundamental domain for $\Omega^\prime$ and has finite measure and
thus $\Omega^\prime$ gives a partial embedding $H \preceq_\Sigma \Lambda$.
\end{proof}

\section{Definition and Basic Properties of Bi-exactness}\label{Section; Bi-exactness}
We recall the definition and basic properties of bi-exactness. This notion was introduced in the $15$th chapter of Brown and Ozawa's book \cite{Brown--Ozawa; Approximation}. This section entirely relies on that book.
\begin{definition}
A subset $\Gamma_1$ of $\Gamma$ is said to be \textbf{small relative to} $\G$ if
there exist $s_1, t_1, \cdots, s_n, t_n \in \Gamma$ and $\Lambda_1, \cdots, \Lambda_n \in \G$
satisfying
$\Gamma_1 \subset \bigcup_{i = 1}^n s_i \Lambda_i t_i$.

\end{definition}
Let $c_0(\Gamma; \G)$ be a $C^*$-subalgebra of $\ell_\infty \Gamma$ generated by functions
whose supports are small relative to $\G$.
\begin{definition}
The group $\Gamma$ is said to be \textbf{bi-exact relative to} $\G$ if there exists a map $\mu : \Gamma \rightarrow \mathrm{Prob}(\Gamma) \subset \ell_1 \Gamma$,
with the property that for any $\epsilon$ and $s, t \in \Gamma$,
there exists a small subset $\Gamma_1$ relative to $\G$ such that
\begin{eqnarray*}
    \| \mu(s x t) - s \mu(x) \|_1 < \epsilon, \quad x \in \Gamma \cap \Gamma_1 ^c.
\end{eqnarray*}
\end{definition}

The following is a useful characterization of bi-exactness.

\begin{proposition}[Proposition 15.2.3 in Brown--Ozawa \cite{Brown--Ozawa; Approximation}]
\label{Proposition; Characterization of Bi-exactness by Actions on Boundaries}
The group $\Gamma$ is bi-exact relative to $\G$ if and only if the Gelfand spectrum of
$\ell_\infty \Gamma / c_0(\Gamma; \G)$ is amenable as a $\Gamma \times \Gamma$-space with
the left-times-right translation action.
\end{proposition}

\begin{remark}
The class $\mathcal{S}$ defined in Ozawa's paper \cite{Ozawa; Kurosh} is the same as the set of countable groups $\Gamma$
which are bi-exact relative to $\{ \{1\} \}$.
The Gromov's word hyperbolic groups are in $\mathcal{S}$. Discrete subgroups
of connected simple Lie groups of rank one are in $\mathcal{S}$
$($by using \cite{Higson--Guentner}, \cite{skandalis; nuclearite}$)$.
The class of amenable countable groups is a subclass of $\mathcal{S}$. A wreath product
$A \wr G$ is in $\mathcal{S}$ if $G \in \mathcal{S}$ and $A$ is amenable. The group $\Z^2 \rtimes \mathrm{SL}(2, \Z)$ is in $\mathcal{S}$ $($by Ozawa \cite{Ozawa; Kurosh, Ozawa; An Example}$)$.
\end{remark}

The notion of bi-exactness well behaves under being taken direct product,
wreath product and free product with amenable amalgamation.

\begin{lemma}[Lemma 15.3.3, Lemma 15.3.5 in \cite{Brown--Ozawa; Approximation}]
\label{Lemma; Direct Product}
Let $\Gamma_i \ (1 \le i \le n)$ be
countable groups and let $\Gamma_0$ be an amenable group. We denote by
$\Gamma$ the direct product $\Gamma_0 \times \prod_{i = 1}^n \Gamma_i$.
Let $\G_i$ be a non-empty family of subgroups of $\Gamma_i\ (1 \le i \le n)$
and let $\G$ be the family of subgroups
\begin{eqnarray*}
\G = \bigcup_{i = 1}^n
     \left\{\left. \Gamma_0 \times \Lambda \times \prod_{j \neq i} \Gamma_j \ \right| \ \Lambda \in \G_i \right\}.
\end{eqnarray*}
If $\Gamma_i$ is bi-exact relative to $\G_i$, then $\Gamma$ is bi-exact relative to $\G$.
\end{lemma}

\begin{lemma}[Lemma 15.3.6 in \cite{Brown--Ozawa; Approximation}]
\label{Lemma; Wreath Product}
If $A$ is amenable and $G$ is exact, then the wreath product
$A \wr G$ is bi-exact relative to $\{ G \}$.
\end{lemma}

\begin{lemma}[Lemma 15.3.12 in \cite{Brown--Ozawa; Approximation}]
\label{Lemma; Amalgamated Free Product}
Let $\Gamma_1, \Gamma_2$ be countable groups and $A$ be a common subgroup of
$\Gamma_1, \Gamma_2$. If $\Gamma_1, \Gamma_2$ are exact and $A$ is amenable,
then the amalgamated free product $\Gamma_1 \ast_{A} \Gamma_2$ is bi-exact relative
to $\{\Gamma_1, \Gamma_2\}$.
\end{lemma}

\section{Location of Subgroups}\label{Section; Location of Subgroups}

The goal of this section is Theorem \ref{Theorem; Non embeddability implies amenability}, which is a consequence of

\begin{proposition}\label{Proposition; NonEm implies Ame in SOE setting}
Let $H$ be a subgroup of $G$ and $\Gamma$ be bi-exact relative to $\mathcal{G}$. Let $\beta$ be a free m.p.\ action of $\Gamma$
on a standard measure space $(Y, \mu)$
and let $\alpha$ be a free m.p.\,action of $G$
on a measurable subset $X \subset Y$ with measure $1$.
Suppose that $\alpha(G) (x) \subset \beta(\Gamma) (x)$, for a.e.~$x \in X$.
We regard the infinite measure space
$\Sigma = \R_\beta \cap (X \times Y)$ as a measurable embedding of $G$ into $\Gamma$, on which
$G$ acts on the first entry and $\Gamma$ acts on the second entry.
If for any $\Lambda \in \G$, there exists no partial embedding of $H$ into $\Lambda$ in $\Sigma$,
then the centralizer $Z_{G}(H)$ is amenable.
\end{proposition}

Before starting the proof of Proposition \ref{Proposition; NonEm implies Ame in SOE setting},
we fix some notations and prove a $C^*$-algebraical continuity property for $\Gamma$-action on $Y$. The notations are similar to those in Sako \cite{Sako}, but we write again for the self-containment.
The action $\beta$ (resp.~$\alpha$) gives a group action of $\Gamma$ (resp.~$G$)
on $L^\infty(Y)$ (resp.~$L^\infty(X)$).
We use the same notation $\beta\ ({\rm resp.}\ \alpha)$ for this action. Let $p \in L^\infty(Y)$ be
the characteristic function of $X$. The algebra $L^\infty(Y)$ and the group $\Gamma$ are represented
on $L^2(\R_\beta, \nu)$ as
\begin{eqnarray*}
    (f \xi ) (x, y) &=& f(x) \xi(x, y), \quad f \in L^\infty(Y),\\
    (u_{\gamma} \xi) (x, y) &=& \xi(\beta_{\gamma^{-1}}(x), y),
     \quad \gamma \in \Gamma,\ \xi \in L^2(\R_\beta),\ (x, y) \in \R_\beta.
\end{eqnarray*}
We denote by $B$ the $C^\ast$-algebra generated by the images, which is the reduced crossed product algebra
$B = L^\infty(Y) \rtimes_{\rm red} \Gamma$.
Its weak closure is the group measure space construction
$\M = L^\infty(Y) \rtimes \Gamma$ (Murray and von Neumann \cite{Murray--vN; IV}).
We denote by $\tr$ the canonical faithful normal semi-finite trace on $\M$.
The unitary involution $J$ of $(\M, \tr)$ is written as
\begin{eqnarray*}
    (J \xi) (x,y) = \overline{\xi(y, x)}, \quad \xi \in L^2(\R_\beta),\ (x, y) \in \R_\beta.
\end{eqnarray*}

The group $G$ is represented on
$p L^2(\mathcal{R}_\beta) = L^2(\mathcal{R}_\beta \cap (X \times Y))$ by
\begin{eqnarray*}
    ( v_g \xi ) (x, y) &=& \xi(\alpha_{g^{-1}}(x), y),
                \quad g \in G,\ \xi \in p L^2(\mathcal{R}_\beta).
\end{eqnarray*}
We denote by $C^*_\lambda (G)$ the $C^\ast$-algebra generated by these operators.
The algebra is isomorphic to the reduced group $C^*$-algebra of $G$.
The Hilbert space $L^2(\R_\alpha, \nu)$ can be
identified with a closed subspace of $p L^2(\R_\beta)$.
The algebra $C^*_\lambda (G)$ is also represented on $L^2(\R_\alpha)$ faithfully.
We denote by $P$ the orthogonal projection from $L^2(\R_\beta)$ onto $L^2(\R_\alpha)$.
We note that the algebra $p B p$ does not contain $C^*_\lambda (G)$ in general,
although there exists an inclusion between their weak closures.

Let $e_\Delta$ be the projection from $L^2(\R_\beta)$ onto the set of $L^2$-functions
supported on the diagonal subset of $\R_\beta$. This is the Jones projection for $L^\infty(Y) \subset \M$.
Consider $L^\infty(\R_\beta) \subset \mathcal{B}(L^2(\R_\beta))$ by multiplications. For $\gamma \in \Gamma$ and a subset $\Gamma_0 \subset \Gamma$, we define the projections
$e(\gamma), e(\Gamma_0)$ by
\begin{eqnarray*}
    e(\gamma) = J u_\gamma J e_\Delta J u_\gamma^* J
              , \quad
    e(\Gamma_0) = \sum_{\gamma \in \Gamma_0} e(\gamma) \in L^\infty(\R_\beta).
\end{eqnarray*}
For $g \in G$ and a subset $G_0 \subset G$, we define the projections $f(g), f(G_0)$ by
\begin{eqnarray*}
    f(g) = v_g e_\Delta v_g^* = v_g (P e_\Delta) v_g^*
         , \quad
    f(G_0) = \sum_{g \in G_0} f(g) \in L^\infty(\R_\beta \cap (X \times Y)).
\end{eqnarray*}

Let $K \subset \B(L^2(\R_\beta))$ be the hereditary subalgebra of $\B(L^2(\R_\beta))$
with approximate units $\{ e(\Gamma_0) \ | \ \Gamma_0 {\rm \ is\ small\ relative\ to\ } \mathcal{G} \}$, that is,
\begin{eqnarray*}
    K =
    \overline{\bigcup_{\Gamma_0} e(\Gamma_0) \B(L^2(\R_\beta)) e(\Gamma_0)}^{\ \| \cdot \|}.
\end{eqnarray*}
The algebras $B$ and $J B J$ are in the multiplier of $K$, so is $D = C^\ast(B, JBJ)$.

The algebra $B$ satisfies the following continuity property.
The proof is conceptually identical
to Proposition 4.2 of Ozawa's paper \cite{Ozawa; Kurosh}.

\begin{proposition}\label{Proposition; ME Continuity for Gamma}
The following map is continuous with respect to the minimal tensor norm:
\begin{eqnarray*}
    \Psi \colon B \otimes_\mathbb{C} J B J \ni \sum_{i=1}^k b_i \otimes J c_i J
    \mapsto \sum_{i=1}^k b_i J c_i J + K \in (D + K) / K.
\end{eqnarray*}
\end{proposition}

In the case of $\mu(Y) < \infty$,
if $\Psi$ were continuous without taken quotient by $K$, this condition would deduce
amenability on the group $\Gamma$. The above Proposition can be regarded as
a weakened amenability property for the $\Gamma$-action.
We prove the above by using an assist of $\ell_\infty \Gamma / c_0 (\Gamma; \G)$. A property of topological amenability proved by C. Anantharaman-Delaroche \cite{Anantharaman-Delaroche} plays a vital role.
In the proof, ``$\otimes$'' stands for the minimal tensor of $C^*$-algebras.

\begin{proof}
Define a representation $m_\cdot$ of $\ell_\infty \Gamma$ on
$L^2 (\R_\beta)$ by the multiplication
\begin{eqnarray*}
    [m_\phi (\xi)] (\gamma x, x) =
    \phi(\gamma) \xi (\gamma x, x),
    \quad \xi \in L^2 \R_\beta,
    \gamma \in \Gamma, \phi \in \ell_\infty \Gamma.
\end{eqnarray*}
Let $\widetilde{D}$ be the
$C^\ast$-algebra generated by $D$ and the image of $m$.
It is easy to see that $\widetilde{D}$ is in the multiplier of $K$.
The preimage $m^{-1}(m(\ell_\infty \Gamma) \cap K)$ is $c_0(\Gamma; \G)$.
The homomorphism $m$ also gives an injective homomorphism of $\ell_\infty \Gamma
/ c_0(\Gamma; \G)$ into $(\widetilde{D} + K) / K$.

Let $E$ be the minimal tensor product
$
E = L^\infty Y \otimes J L^\infty Y J \otimes
\ell_\infty \Gamma /c_0(\Gamma, \G).
$
The product group $\Gamma \times \Gamma$ acts on $E$ by
\begin{eqnarray*}
    & &  \mathfrak{A}(g, h)(f_1 \otimes J f_2 J \otimes (\phi + c_0(\Gamma; \G)))\\
    &=&  \beta_g(f_1) \otimes J \beta_h(f_2) J \otimes (l_g r_h (\phi) + c_0(\Gamma; \G)),
\end{eqnarray*}
where $l, r$ stand for the left and the right translation
actions on $\ell_\infty \Gamma / c_0(\Gamma; \G)$ respectively.
Let $\widetilde{E}$ be the reduced crossed product
$E \rtimes_{\rm red} (\Gamma \times \Gamma)$.

We claim that there exists a $\ast$-homomorphism $\Psi \colon \widetilde{E} \rightarrow (\widetilde{D} + K) / K$ satisfying
\begin{eqnarray*}
    \Psi (f_1 \otimes J f_2 J \otimes (\phi + c_0(\Gamma; \G))) = f_1 J f_2 J m_\phi + K,
    \quad \\
    \Psi (g, h) = u_g J u_h J + K,
    \quad f_1, f_2 \in L^\infty Y,\ \phi \in \ell_\infty\Gamma,\ (g, h) \in \Gamma \times \Gamma.
\end{eqnarray*}
We consider the $\ast$-homomorphism from $L^\infty Y \otimes_\mathbb{C}
J L^\infty Y J \otimes_\mathbb{C} \ell_\infty \Gamma / c_0(\Gamma; \G)$
to $(\widetilde{D} + K) / K$ given by the first equation.
Since $L^\infty Y, J L^\infty Y J$ are nuclear by Takesaki's theorem \cite{Takesaki},
this homomorphism extends to the minimal tensor product $E$.
The homomorphism $\Gamma \times \Gamma
\ni (g, h) \mapsto u_g J u_h J + K \in (\widetilde{D} + K) / K$
gives the covariant system of the action $\mathfrak{A}$, that is,
\begin{eqnarray*}
    & & (u_g J u_h J + K)
        \Psi(f_1 \otimes J f_2 J \otimes (\phi + c_0(\Gamma; \G)))
        (u_g J u_h J + K)^\ast\\
    &=& u_g f_1 u_g^\ast J u_h f_2 u_h^\ast J m(l_g r_h(\phi)) + K\\
    &=& \Psi(\beta_g(f_1) \otimes J \beta_h(f_2) J \otimes (l_g r_h(\phi) + c_0(\Gamma; \G))).
\end{eqnarray*}
We get a $\ast$-homomorphism $\Psi$ from
the full crossed product $E \rtimes_{\rm full} (\Gamma
\times \Gamma)$ to $(\widetilde{D} + K) / K$.

The subalgebra $\mathbb{C} \otimes \mathbb{C} \otimes \ell_\infty \Gamma / c_0(\Gamma, \G)$
is in the center of $E$ and globally invariant under the
action. Since $\Gamma$ is bi-exact relative to $\G$,
the $\Gamma \times \Gamma$-action on the Gelfand spectrum of
$\mathbb{C} \otimes \mathbb{C} \otimes \ell_\infty \Gamma / c_0(\Gamma, \G)$
is amenable (Proposition \ref{Proposition; Characterization of Bi-exactness by Actions on Boundaries}).
 The full crossed product algebra $E \rtimes_{\rm full} (\Gamma \times \Gamma)$ coincides with the reduced
crossed product $\widetilde{E}$, by \cite{Anantharaman-Delaroche}.
The restriction of $\Psi$ on
$(L^\infty Y \otimes J L^\infty Y J) \rtimes_{\rm red}
(\Gamma \times \Gamma) \subset \widetilde{E}$
gives $\Psi$ in Proposition \ref{Proposition; ME Continuity for Gamma}.
\end{proof}

We proceed to prove Proposition \ref{Proposition; NonEm implies Ame in SOE setting}.
The proof says that when the $H$-action on the first entry of $\R_\beta \cap (X \times Y)$ flees all projections $p e(\Gamma_0)$ for small sets $\Gamma_0$, Proposition \ref{Proposition; ME Continuity for Gamma} deduces a continuity property of the reduced group $C^*$-algebra $C^\ast_\lambda (Z_G(H))$.

\begin{proof}
We may assume that the family $\mathcal{G}$ is invariant under conjugation. Indeed, by the definition,
$\Gamma$ is bi-exact relative to $\mathcal{G}$ if and only if $\Gamma$ is bi-exact relative to
$\widetilde{\mathcal{G}} = \bigcup_{\gamma \in \Gamma} \gamma \mathcal{G} \gamma^{-1}$.
If there exists a partial embedding $\Omega \subset \R_\beta \cap (X \times Y)$
of $H$ into $\gamma \Lambda \gamma^{-1}$ for some
$\Lambda \in \mathcal{G}$, then $\gamma^{-1} \Omega$ gives a partial embedding
of $H$ into $\Lambda$. Assume that $\mathcal{G}$ is conjugation invariant.

Denote $G_1 = Z_G(H)$. The unitaries $\{ v_g \ | \ g \in G_1 \}$ gives a faithful representation of
$C^*_\lambda (G_1)$ on $p L^2(\R_\beta) p$. We fix this representation.
We denote $C^\ast_\rho (G_1) = J C^*_\lambda (G_1) J$.
To show the amenability of $G_1$, it suffices to show that the natural homomorphism
\begin{eqnarray*}
    \Phi \colon C^*_\lambda (G_1) \otimes_\mathbb{C} C^\ast_\rho (G_1)
    \longrightarrow \B(p L^2 \M p) = \B(L^2(\R_\beta) \cap (X \times X)),
\end{eqnarray*}
is continuous with respect to the minimal tensor norm.
(See Section 2.6 of \cite{Brown--Ozawa; Approximation}, for example).
We take an arbitrary positive number $\epsilon > 0$, a finite subset
$\mathcal{F} \subset G_1$ and $x \in C^*_\lambda (G_1) \otimes_\mathbb{C}
C^\ast_\rho (G_1)$ of the following form:
\begin{eqnarray*}
    x = \sum_{s, t \in \mathcal{F}} c(s, t) v_s \otimes
    J v_t J, \quad c(s, t) \in \C.
\end{eqnarray*}
Then $\Phi(x)$ is given by
$\Phi(x) = \sum_{s, t \in \mathcal{F}} c(s, t) v_s J v_t J$.

Since the norm of $\Phi(x)$ is almost attained by some vector, there exists
a finite subset $\Gamma_0 \subseteq \Gamma$ satisfying
\begin{eqnarray}\label{Equation; A}
    \| \Phi(x) e(\Gamma_0) \| > \| \Phi(x) \| - \epsilon.
\end{eqnarray}
We claim that there exists $\delta > 0$ with the property: For any
projection $f$ in $L^\infty X$ with $\tr(p - f) \le \delta$,
\begin{eqnarray}\label{Equation; B}
      \| \Phi(x) e(\Gamma_0) f J f J\|
    > \| \Phi(x) e(\Gamma_0) \| - \epsilon.
\end{eqnarray}
Otherwise, there would exist a sequence of projections $\{ f_k \} \subseteq L^\infty X$
such that $\tr(p - f_k) < 2^{-k}$ and
$\| \Phi(x) e(\Gamma_0) f_k J f_k J\| \le \| \Phi(x) e(\Gamma_0) \| - \epsilon$.
Denote $p_k = f_k \wedge f_{k + 1} \wedge \ldots$. Then we get
$\| \Phi(x) e(\Gamma_0) p_k J p_k J\| \le \| \Phi(x) e(\Gamma_0) \| - \epsilon$.
This contradicts the fact that $p_k J p_k J$ is an increasing sequence
converging to $p J p J$.

The unitary $v_s$ can be written as a Fourier expansion
$v_s = \sum_{\gamma} u_\gamma p(s, \gamma)$,
by some projections $\{ p(s, \gamma) \} \subset L^\infty X$ with $\sum_\gamma p(s, \gamma) = p$.
There exists an increasing sequence of projections $\{q_n(s)\} \subset L^\infty X$
such that $\lim_n \tr(q_n(s)) = \tr(p)$ and $v_s q_n(s) \in B = L^\infty Y \rtimes_{\rm red} \Gamma$.
Since $\mathcal{F}$ is a finite set, there exists a projection $q_1 \in L^\infty X$ satisfying
$\tr(p - q_1) \le \delta / 3$
and $v_s q_1 \in B$ for all $s \in \mathcal{F}$.

The operator
$x (q_1 \otimes J q_1 J) = \sum c(s, t) v_s q_1 \otimes J v_t q_1J$
is in the domain of $\Psi$ in Proposition \ref{Proposition; ME Continuity for Gamma} and its image is
\begin{eqnarray*}
    & & \Psi(x (q_1 \otimes J q_1 J))
     =  \sum_{s, t \in G_1} c(s, t) v_s q_1 J v_t q_1J + K
     =  \Phi(x)q_1 J q_1 J + K.
\end{eqnarray*}
Since $\Psi$ is continuous (or equivalently contractive),
we get
\begin{eqnarray*}
          \| x \|_{\rm min}
    &\ge& \| \Psi(x (q_1 \otimes J q_1 J)) \|
      =   \| \Phi(x) q_1 J q_1 J + K \|_{(D + K) / K} \\
     &=&
    \inf \{ \| \Phi(x) q_1 J q_1 J (1 - e(\Gamma_1)) \|
                                \ | \
    \Gamma_1 \subset \Gamma {\rm\ small\ relative\ to\ } \mathcal{G}\}.
\end{eqnarray*}
We used the fact that
$\{ e(\Gamma_1) \ | \
\Gamma_1 \subset \Gamma {\rm \ small\  relative\ to\ } \mathcal{G}\}$
is a net of approximate units for $K$. We
get a finite subset $\Gamma_1 \subset \Gamma$ with
\begin{eqnarray}\label{Equation; C}
      \| x \|_{\rm min} + \epsilon
    > \| \Phi(x) q_1 J q_1 J (1 - e(\Gamma_1)) \|.
\end{eqnarray}
We may assume that $\Gamma_1$ is of the form
$\Gamma_1 = \bigcup_{i = 1}^n \Lambda_i \gamma_i$,
for some $\Lambda_i \in \mathcal{G}$, since $\mathcal{G}$
is conjugation invariant. To show the continuity of $\Phi$, we will show an inequality between the right hand side of (\ref{Equation; C}) and the left hand side of (\ref{Equation; B}) for an appropriate $f$.

Write $\Sigma = \R_\beta \cap (X \times Y)$ and regard $\Sigma$ as a measurable embedding of $G$ into $\Gamma$. We make use of notations in Subsection
\ref{Subsection; Function Valued Measure}. The projection $p e_\Delta$ corresponds to a
$\Gamma$ fundamental domain of $\Sigma$. We identify $X$ with the fundamental domain.
Then the projections $p e(\Lambda_i \Gamma_0), p e(\Lambda_i \gamma_i)$ are written as
\begin{eqnarray*}
    p e(\Lambda_i \Gamma_0)
    &=& \sum_{\lambda \in \Lambda_i, \gamma \in \Gamma_0}
               J u_\lambda u_\gamma J p e_\Delta J u_\gamma^* u_\lambda^*J
     = \chi(\Lambda_i \Gamma_0 X) \in L^\infty \Sigma\\
    p e(\Lambda_i \gamma_i)
    &=& \sum_{\lambda \in \Lambda_i}
               J u_\lambda u_{\gamma_i} J p e_\Delta J u_{\gamma_i}^* u_\lambda^*J
    = \chi(\Lambda_i \gamma_i X) \in L^\infty \Sigma.
\end{eqnarray*}
They are elements in $(L^\infty \Sigma)^{\Lambda_i}$
and their values of $\Tr_i = \Tr_{\Lambda_i}$ are finite. Let $e_0, e_1$ be the projections in
$\widetilde{\A} = (L^\infty \Sigma)^{\Lambda_1} \oplus
\cdots \oplus (L^\infty \Sigma)^{\Lambda_n}$
defined by
\begin{eqnarray*}
    e_0 &=& \chi(\Lambda_1 \Gamma_0 X) \oplus \chi(\Lambda_2 \Gamma_0 X)
            \oplus \cdots \oplus \chi(\Lambda_n \Gamma_0 X),\\
    e_1 &=& \chi(\Lambda_1 \gamma_1 X) \oplus \chi(\Lambda_2 \gamma_2 X)
            \oplus \cdots \oplus \chi(\Lambda_n \gamma_n X).
\end{eqnarray*}
Let $\Tr$ be the trace on $\widetilde{\A}$ given by the summation
$\Tr = \Tr_1 + \Tr_2 + \cdots + \Tr_n$.
The values $\Tr(e_0)$ and $\Tr(e_1)$ are finite.

Let $\mathcal{C} \subseteq \widetilde{\A} \cap L^2(\widetilde{\A}, \Tr)$
 be the set of convex combinations
\begin{eqnarray*}
{\rm conv}\{ h (e_1) = \chi(h \Lambda_1 \gamma_1 X) \oplus \chi(h \Lambda_2 \gamma_2 X) \oplus \cdots \oplus
            \chi(h \Lambda_n \gamma_n X) \ | \ h \in H \}.
\end{eqnarray*}
We take the unique element $x = x_1
\oplus x_2 \oplus \cdots \oplus x_n$ with the smallest $2$-norm in
$2$-norm closure $\overline{\mathcal{C}}$. Since the set $\overline{\mathcal{C}}$
is globally fixed under the action of $H$, $x$ is fixed under the action of $H$.
Since $x$ is a $L^2$-limit of positive functions, $x$ is positive.
For $t > 0$, its preimage $\Omega_t = \bigsqcup_{i =1}^n \Omega_{i, t} \subset \Sigma \times \{1,2,\cdots, n\}$ of
$[t, \infty)$ has a finite value of $\Tr$.
Since $i$-th entry of every element $y \in \mathcal{C}$ is $\Lambda_i$-invariant, so is $x$.
The $i$-th measurable subset
$\Omega_{i,t} \subset \Sigma$ is $H$-invariant and $\Lambda_i$-invariant, and the measure of its
$\Lambda_i$ fundamental domain is finite.
The assumption of Proposition \ref{Proposition; NonEm implies Ame in SOE setting} tells that
$\Omega_i$ is a null set. This means that $e_{[t, \infty)} = 0$ and
thus we get $x = 0 \in \overline{\mathcal{C}}$.
Since the elements of the form $k^{-1} \sum_{i = 1}^k h_i(e_1)$ is
$2$-norm dense in $\overline{\mathcal{C}}$,
there exist $h_1, h_2,
\ldots, h_k \in G$ satisfying
\begin{eqnarray*}
    \Tr \left( \frac{1}{k} \sum_{i = 1}^k h_i(e_1) e_0 \right) \le
    \delta / 3.
\end{eqnarray*}
We choose $h \in \{h_1, h_2, \ldots, h_k\}$ satisfying
$\Tr( h(e_1) e_0) \le \delta / 3$.

Let $\E^{(i)}_X$ be the function valued measure from
$\widehat{(L^\infty \Sigma)^{\Lambda_i}_+}$ to $\widehat{L^\infty X_+}$ defined in
Subsection \ref{Subsection; Function Valued Measure}.
Each measurable function
$\E^{(i)}_X (h \Lambda_i \gamma_i X \cap \Lambda_i \Gamma_0 X)$
is integer valued on $X$.
The function
$F = \sum_{i = 1}^n \E^{(i)}_X (h \Lambda_i \gamma_i X \cap \Lambda_i \Gamma_0 X )$
is also integer valued. Let $p - q_2 \in L^\infty X$ be the support of $F$. It follows that
\begin{eqnarray*}
    \tr(p - q_2) \le
    \int_X F d \mu = \Tr( h(e_1) e_0) \le
    \delta / 3.
\end{eqnarray*}
Since $q_2 \E^{(i)}_X (h \Lambda_i \gamma_i X \cap \Lambda_i \Gamma_0 X) = 0$,
we also get
\begin{eqnarray*}
    \chi(h \Lambda_i \gamma_i X) \chi(\Lambda_i \Gamma_0 X) q_2 =
    v_h e(\Lambda_i \gamma_i) v_h^* e(\Lambda_i \Gamma_0) q_2 = 0.
\end{eqnarray*}
Since $e(\Lambda_i \Gamma_0) q_2 = q_2 e(\Lambda_i \Gamma_0)$, it follows that
\begin{eqnarray*}
            v_h e(\Lambda_i \gamma_i) v_h^*
    &\perp& q_2 e(\Lambda_i \Gamma_0),\\
            v_h e(\Gamma_1) v_h^*
       =    \bigvee_{i = 1}^n v_h e(\Lambda_i \gamma_i) v_h^*
    &\perp& \bigwedge_{i = 1}^n q_2 e(\Lambda_i \Gamma_0)
      \ge   q_2 e(\Gamma_0),\\
            v_h (1 - e(\Gamma_1)) v_h^*
    &\ge&   q_2 e(\Gamma_0).
\end{eqnarray*}

Since $[v_s, v_h] = 0$ for $s \in G_1$, letting $f = \alpha_h(q_1) q_1 q_2$,
\begin{eqnarray*}\label{Equation; D}
        \| \Phi(x) q_1 J q_1 J (1 - e(\Gamma_1)) \|
    &=& \| v_h \Phi(x) q_1 J q_1 J (1 - e(\Gamma_1)) v_h^* \|\\
    &=& \| \Phi(x) \alpha_h(q_1) J q_1 J v_h (1 - e(\Gamma_1)) v_h^* \|\\
  &\ge& \| \Phi(x) \alpha_h(q_1) J q_1 J q_2 e(\Gamma_0) \|\\
  &\ge& \| \Phi(x) e(\Gamma_0) f J f J \|.
\end{eqnarray*}
Since $\tr(p - f) \le \tr(p - \alpha_h(q_1)) + \tr(p - q_1) + \tr(p - q_2) \le \delta$, we can use the equation (\ref{Equation; B}).
Combining the above inequality,  (\ref{Equation; A}), (\ref{Equation; B}) and (\ref{Equation; C}),
we get
\begin{eqnarray*}
    \| x \|_{\rm min} + 3 \epsilon > \| \Phi(x) \|.
\end{eqnarray*}
Since the positive number $\epsilon$ is arbitrary,
we get the desired continuity of $\Phi$ and Proposition \ref{Proposition; NonEm implies Ame in SOE setting}.
\end{proof}

The following is a key result in this paper, which deduces three types of
results on direct product groups, wreath product groups and amalgamated free products.
\begin{theorem}\label{Theorem; Non embeddability implies amenability}
Let $\Gamma$ be a countable group which is bi-exact relative to $\mathcal{G}$
and let $H \subset G$ be an inclusion of countable groups.
Suppose that there exists an ergodic measurable embedding $\Sigma$
of $G$ into $\Gamma$ and that $\Sigma_H \subset \Sigma$ is an $H \times \Gamma$-invariant non-null measurable subset.

If the centralizer $Z_G(H)$ of $H$
is non-amenable,
then there exists a partial embedding $\Omega$ of $H$ into $\Lambda$
satisfying $\Omega \subset \Sigma_H$.
In particular, if $G \preceq_{\rm ME} \Gamma$ and $Z_G(H)$ is non-amenable, then
$H \preceq_{\rm ME} \Lambda$ for some $\Lambda \in \mathcal{G}$.
\end{theorem}

\begin{proof}
Let $\Sigma$ be an arbitrary ergodic measurable embedding of $G$ into $\Gamma$.
We denote by $\hat{G}$ the subgroup of $G$ generated by $H$ and $Z_G(H)$.
Let $\Sigma_H \subset \Sigma$ be a non-null measurable subset
invariant under $H \times \Gamma$.
To show that there exists a partial embedding of $H \preceq_\Sigma \Lambda \in \G$ in
$\Sigma_H$,
we have only to find a partial embedding $\Omega$
in $\Sigma_1 = \bigcup \{g \Sigma_H \ | \ g \in \hat{G}\}$.
Suppose that $Z_G(H)$ is non-amenable.

First we consider the case of $[\Gamma : G]_\Sigma \ge 1$.
We take a standard probability space $(X^\prime, \mu)$ which is equipped
with a weakly mixing free measure preserving $G$-action.
Let $\Gamma$ act on $X^\prime$ trivially. We regard
$\Sigma^{\mathrm{free}} = \Sigma \times X^\prime$ as
a measurable embedding, on which $G$ and $\Gamma$ act by diagonal actions respectively.
Since the $G$-action on the set
$\Gamma \backslash \Sigma^{\mathrm{free}} \cong (\Gamma \backslash \Sigma) \times X^\prime$ is free and ergodic,
$\Sigma^{\mathrm{free}}$ is an ergodic measurable embedding coming form SOE. The coupling index
is $[\Gamma : G]_{\Sigma^{\mathrm{free}}} = [\Gamma : G]_\Sigma \ge 1$.
There exist a $\Gamma$-action $\beta$ on a standard measure space $Y$,
a measurable subset $X \subset Y$ and a $G$-action $\alpha$ on a standard probability space $X$
such that
$\Sigma^{\mathrm{free}} \cong \R_\beta \cap (X \times Y)$.
The measurable subset $\Sigma^{\mathrm{free}}_1 = \Sigma_1 \times X^\prime \subset \Sigma^{\mathrm{free}}$
is a measurable embedding of $\hat{G}$ into $\Gamma$.
Since $\Sigma^{\mathrm{free}}_1$ is $\Gamma$-invariant, $\Sigma^{\mathrm{free}}_1 = \R_\beta \cap (X_1 \times Y)$
for some $\hat{G}$-invariant measurable subset $X_1 \subset X$.
We apply the contrapositive of Proposition \ref{Proposition; NonEm implies Ame in SOE setting} for $\alpha |_{\hat{G}} \colon \hat{G} \curvearrowright X_1$ and $\beta \colon \Gamma \curvearrowright Y$.
We get some $\Lambda \in \mathcal{G}$ and an $H \times \Lambda$-invariant measurable subset $\Omega^{\mathrm{free}}_1 \subset \Sigma^{\mathrm{free}}_1$
so that the measure of a $\Lambda$ fundamental domain of $\Omega^{\mathrm{free}}_1$ is finite.

We define the measurable function $\phi$ on $\Sigma_1$ by
\begin{eqnarray*}
    \phi(s) = \mu(\{x \in X^\prime \ | \ (s, x) \in \Sigma_1 \times X^\prime = \Omega^{\mathrm{free}}_1 \}),
\end{eqnarray*}
which is defined almost everywhere on $s \in \Sigma_1$. The function $\phi$ is invariant under
the $H$-action and $\Lambda$-action on $\Sigma_1$ outside a null set.
Take a fundamental domain
$D_1 \subset \Sigma_1$ for the $\Lambda$-action on $\Sigma_1$.
Since $\Omega^{\mathrm{free}}_1 \cap (D_1 \times X^\prime)$
is the $\Lambda$-fundamental domain of $\Omega^{\mathrm{free}}_1$ and has finite measure,
the function $\phi |_{D_1}$ is integrable, by Fubini's Theorem.
Any non-trivial level set of $\phi$ gives a partial embedding of $H$ into $\Lambda$ in $\Sigma_1$.

We consider the case of $[\Gamma : G]_\Sigma < 1$. We take an integer $n$ with
$n [\Gamma : G]_\Sigma \ge 1$. We define
$\widetilde{\Gamma} = \Gamma \times \Z / n\Z$ and
$\widetilde{\Sigma} = \Sigma \times \Z / n \Z$.
Let $\widetilde{\Gamma}$ act on $\widetilde{\Sigma}$ by the product action and $G$ act
on $\Z / n \Z$ trivially.
We note that
$\widetilde{\Gamma}$ is bi-exact relative to $\mathcal{G} \times  \{ 1 \}$.
Since
$[\widetilde{\Gamma} : G]_{\widetilde{\Sigma}} = n [\Gamma : G]_\Sigma \ge 1$,
by the above argument
there exist $\Lambda \in \mathcal{G}$ and a partial embedding $\widetilde{\Omega} \subset \widetilde{\Sigma}$ of $H$ into $\Lambda \times \{1\}$.
Then we define a non-null subset $\Omega \subset \Sigma$ by a non-null
$\Omega \times \{k\}
= (\Sigma \times \{k\}) \cap \widetilde{\Omega}$.
This measurable subset gives an embedding of $H$ into $\Lambda$.
\end{proof}

\section{Factorization of Product Groups}
\label{Section; Factorization of Product Groups}
Before stating main theorems in this section,
we remark some general fact (Proposition \ref{Proposition; ME and MEm passes to Quotient Groups}) on partial embeddings of normal subgroups.
\subsection{ME Coupling between Quotient Groups}
Let $(\A, \Tr)$ be a pair of an abelian von Neumann algebra and its faithful
normal semi-finite trace. Let $\Gamma$ be a countable group acting on $\A$ in
trace preserving way. We do not need a condition on freeness.
The following notation will be useful.
\begin{definition}\label{Definition; Fundamental Pair}
A pair $(f, \Lambda_f)$ of a non-zero projection $f \in \A$ and a subgroup $\Lambda_f \subset \Gamma$ is said to be
\textbf{a fundamental pair} if the following conditions hold:
\begin{enumerate}
\item
The projection $f$ is an absolute invariant projection of the $\Lambda_f$-action,
namely, for any projection $f^\prime \le f$ in $\A$ and $\lambda \in \Lambda_f$,
we have $\lambda(f^\prime) = f^\prime$.
\item
For any $\gamma \in \Gamma \cap (\Lambda_f) ^c$, the projection $\gamma (f)$ is
orthogonal to $f$.
\item
The projection $\bigvee_{\gamma \in \Gamma} \gamma(f)$ is $1$.
\end{enumerate}
\end{definition}

Let $\Gamma_\mathrm{nor}$ be the normalizing subgroup for $\Lambda_f$;
$\Gamma_\mathrm{nor} = \{ \gamma \in \Gamma \ | \ \gamma \Lambda_f \gamma^{-1} = \Lambda_f \}$.
Then the group $\Gamma_\mathrm{nor} / \Lambda_f$ naturally acts on $\A q$, where $q$ is the projection
$q = \bigvee_{\gamma \in \Gamma_\mathrm{nor}} \gamma(f)$.
The group $\Lambda_f$ acts on $\A q$ trivially.
If we consider $\A q$ as an $L^\infty$ function space, a measurable subset
corresponding to $f$ is a fundamental domain for the $\Gamma_\mathrm{nor} / \Lambda_f$-action on $\A q$.

\begin{lemma}\label{Lemma; Find a Fundamental Pair}
Let $H \subset G$, $\Lambda \subset \Gamma$ be normal subgroups
and let $(\Sigma, \nu)$ be a standard measure space on which an ergodic $G \times \Gamma$-action is given.
Suppose that the $\Gamma$-action on $\Sigma$ has a fundamental domain $X \subset \Sigma$.

If there exists an $H \times \Lambda$-invariant projection $e \in L^\infty \Sigma$ with $\mathrm{range}(\E^\Lambda_X (e)) \not\subset \{0, \infty\}$,
then there exist an $H \times \Lambda$-invariant
projection $f$ and an intermediate subgroup $\Lambda \subset \Lambda_f \subset \Gamma$
such that
$[\Lambda_f \colon \Lambda] < \infty$ and that
the pair $(f, \Lambda_f / \Lambda)$ is a fundamental pair for the
$\Gamma / \Lambda$-action on $(L^\infty \Sigma)^{H \times \Lambda}$.
\end{lemma}

Before the proof, we note that the action of $\Gamma$ on $L^\infty \Sigma$
globally fixes the fixed point subalgebras $(L^\infty \Sigma)^\Lambda$,
$(L^\infty \Sigma)^{H \times \Lambda}$,
since $\Lambda$ is a normal subgroup of $\Gamma$. Furthermore,
this action preserves the trace $\Tr_\Lambda$ defined in Subsection
\ref{Subsection; Function Valued Measure}.
This is because the definition of $\Tr_\Lambda$
does not depend on the choice of a $\Lambda$ fundamental domain of $\Sigma$.

\begin{proof}
Let $k$ be the minimal element among the positive integers
\begin{eqnarray*}
\bigcup \{
\mathrm{range}(\E^\Lambda_X(e)) \ | \
e \in (L^\infty \Sigma)^{H \times \Lambda}\} \cap \{0, \infty \}^c.
\end{eqnarray*}
We assume $k \in \mathrm{range}(\E^\Lambda_X(e))$.
Let $U \subset X$ be the preimage of $k$.
We replace $e$ with the restriction $e \chi(\Gamma U)$.
Since the subset $U$ is invariant under
the $H$-action on $X \cong \Gamma \backslash \Sigma$,
the restriction is also $H \times \Lambda$-invariant.
Then $\E^\Lambda_X(e)$ is non-zero and $\mathrm{range}(\E^\Lambda_X(e)) \subset \{0, k\}$.
Let $\Omega$ be a measurable subset corresponding to $e$.
There exists a non-null measurable
subset $X_1 \subset U$ such that
\begin{eqnarray*}
\Omega \cap \Gamma X_1 = \bigsqcup_{\gamma_i \in \Gamma_0} \Lambda \gamma_i X_1,
\end{eqnarray*}
for some finite subset $\Gamma_0 = \{\gamma_1, \gamma_2, \cdots, \gamma_k\}$.
By replacing $X$ with $\gamma_1 X_1 \sqcup (X \cap (X_1) ^c)$, we may assume
that $1 = \gamma_1$.
Then the union of $k$-cosets
$\Lambda_f = \bigsqcup_{\gamma_i \in \Gamma_0} \Lambda \gamma_i$
is a subgroup of $\Gamma_1$. Indeed, for $\gamma \in \Gamma$, we get
\begin{eqnarray*}
\E^\Lambda_X (\gamma(e) e) 1_{X_1}
= |\Lambda \backslash (\gamma \Lambda_f \cap \Lambda_f)| 1_{X_1}.
\end{eqnarray*}
Since the projection $\gamma(e) e$ is also $H \times \Lambda$-invariant,
by the minimality of $k$, it follows that
$|\Lambda \backslash (\gamma \Lambda_f \cap \Lambda_f)|
     = k {\rm \ or \ } 0$.
In other words,
we get $\gamma \Lambda_f \cap \Lambda_f = \Lambda_f$ or $\emptyset$. It follows
that $\Lambda_f$ is a subgroup of $\Gamma$.
We define $f$ by $\bigwedge_{\gamma \in \Lambda_f} \gamma(e)$.
Since $\chi(\cup_{i = 1}^n \Lambda \gamma_i X_1) \le f \le e$, the projection $f$
satisfies $\mathrm{range}(\E^\Lambda_X(f)) \subset \{0, k\}$ and
\begin{eqnarray}\label{Equation; minimal pair}
    \gamma(f) =     f \ (\gamma \in \Lambda_f), \quad
    \gamma(f) \perp f \ (\gamma \in \Gamma \cap (\Lambda_f)^c).
\end{eqnarray}

Furthermore, there exists a projection $f$ with the property $(4)$ and
$\E^\Lambda_X (f)$ is $k 1_X$. Let $\alpha$ be the $G$-action on $X$ defined
by the natural identification $X \cong \Gamma \backslash \Sigma$. Since the $G \times \Gamma$-action
on $\Sigma$ is ergodic, the dot action $\alpha \colon G \curvearrowright X \cong \Gamma \backslash \Sigma$ is also ergodic.
Let $V \subset X$ be the support of $\E^\Lambda_X(f)$. This
is $H$-invariant. If $V$ is not $X$, then there exists $g \in G$ such that
$W = V \cap (\alpha_{g^{-1}} (V))^c$ is not null and $H$-invariant.
Then the projection $f + g(f \chi(\Gamma W))$ is also $H \times \Lambda$-invariant.
By Lemma \ref{Lemma; Function Valued Measure}, the value of $\E^\Lambda_X$ is
\begin{eqnarray*}
\E^\Lambda_X(f + g(f \chi(\Gamma W))) = k \chi(V) + k \chi(\alpha_g(W)).
\end{eqnarray*}
We get a projection greater than the original one with the same properties.
By the maximality argument,
we get an $H \times \Lambda_f$-invariant
projection $f$ with $\E^\Lambda_X (f) = k 1_X$.

The $\Lambda_f / \Lambda$-action on $f (L^\infty \Sigma)^{H \times \Lambda}$
is trivial. Indeed, by the minimality of $k$, if a projection $f^\prime$ is smaller than $f$
and $H \times \Lambda$-invariant, then $\mathrm{range}(\E^\Lambda_X(f^\prime)) \subset \{0, k\}$. The projection $f^\prime$ must be written as $f^\prime = f \chi(\Gamma D)$ by some $D \subset X$.
The projection $f^\prime$ is also $\Lambda_f$-invariant.
Since the support of $\E^\Lambda_X (f)$ is $X$, the projection
$\bigvee_{\gamma \in \Gamma} \gamma(f)$ is $1$. It turns out that $(f, \Lambda_f / \Lambda)$
is a fundamental pair for the $\Gamma / \Lambda$-action
on $(L^\infty \Sigma)^{H \times \Lambda}$.
\end{proof}

\begin{proposition}\label{Proposition; ME and MEm passes to Quotient Groups}
Let $H \subset G$, $\Lambda \subset \Gamma$ be normal subgroups of countable groups
and let $(\Sigma, \nu)$ be an ergodic ME coupling for $G$ and $\Gamma$
$($resp.\ an ergodic measurable embedding of $G$ into $\Gamma$$)$.
If there exists a partial embedding from $H$ into $\Lambda$ in $\Sigma$ and if
there exists an $H \times \Lambda$-invariant projection $f \in L^\infty \Sigma$ with
$\mathrm{range}(\mathfrak{E}^H_Y (f)) \not\subset \{0, \infty\}$,
then $G / H \sim_\mathrm{ME} \Gamma/\Lambda$
$($resp. $G / H \preceq_\mathrm{ME} \Gamma/\Lambda$$)$.
\end{proposition}

\begin{proof}
Let $\Omega \subset \Sigma$ be a partial embedding of $H$ into $\Lambda$.
The measurable function $\E^\Lambda_X(\Omega)$ on a $\Gamma$ fundamental domain $X$ is integrable, since
$\int_X \E^{\Lambda}_X(\Omega) = \Tr_\Lambda (\Omega) < \infty$.
Thus there exists a fundamental pair $(e, \Lambda_f / \Lambda)$ for the $\Gamma / \Lambda$-action
on $(L^\infty \Sigma)^{H \times \Lambda}$ by Lemma \ref{Lemma; Find a Fundamental Pair}.
There also exists a fundamental pair $(f, H_f / H)$ for the $G / H$-action
on $(L^\infty \Sigma)^{H \times \Lambda}$ by the other assumption.
Replacing $(f, H_f / H)$ with $(g f, g H_f g^{-1})$, we assume that $e f \neq 0$.

We have two faithful traces $\Tr_\Lambda$ and $\Tr_H$
on the algebra $(L^\infty \Sigma)^{H \times \Lambda}$. We can consider that $(L^\infty \Sigma)^{H \times \Lambda}$ is an $L^\infty$-function space on a standard measure space.
Let $F$ be the Radon-Nikodym derivative $d \Tr_\Lambda / d \Tr_H$.
Since $0 < \Tr_\Lambda(e f) \le \Tr_\Lambda(e) < \infty$,
the function $F$ is integrable on $e f$.
Since both of the traces are invariant under the action of $G$ and $\Gamma$,
the function $F$ is invariant under the action of $G \times \Gamma$.
Thus $d \Tr_\Lambda / d \Tr_H$ is constant $c$.
It turns out that
\begin{eqnarray}\label{Equation; Radon--Nycodym derivative}
\Tr_H(e) = c^{-1} \Tr_\Lambda(e) < \infty.
\end{eqnarray}

Let $\Gamma_\mathrm{nor} \subset \Gamma$ be the normalizing subgroup of $\Lambda_f$.
Let $q \in (L^\infty \Sigma)^{H \times \Lambda}$ be the projection given by
$\bigvee_{\gamma \in \Gamma_\mathrm{nor}} \gamma(e)$.
The group $\Lambda_f$ acts trivially on the algebra
$q (L^\infty \Sigma)^{H \times \Lambda}$.
For $\gamma \in \Gamma \cap (\Gamma_\mathrm{nor})^c$, there exists $\gamma^\prime \in \Lambda_f$
such that $\gamma^{-1} \gamma^\prime \gamma \notin \Lambda_f$.
Then the projections $\gamma(e)$ and
$\gamma^\prime \gamma(e) = \gamma \gamma ^{-1} \gamma^\prime \gamma(e)$
are perpendicular.
It follows that $q$ can be characterized as the largest projection
in $(L^\infty \Sigma)^{H \times \Lambda}$ so that
$\Lambda_f$ acts trivially on $q (L^\infty \Sigma)^{H \times \Lambda}$.
Thus the projection $q$ is invariant under the $G \times \Gamma_\mathrm{nor}$-action. It follows that
there exists a $\Gamma_\mathrm{nor}$-invariant measurable subset
$Y_f \subset Y$ such that $\chi(G Y_f) = q$.

Choose representatives $\{ \gamma_\iota \}_{\iota \in I}$ for the left cosets
$\Gamma / \Gamma_\mathrm{nor}$. Then
the projections $\{ \gamma_\iota (q) \}_{\iota \in I}$ gives a partition of $1_Y$.
The projection $\gamma_\iota (q)$ is the characteristic function of
$\beta(\gamma_\iota )( Y_f ) \subset Y$.
Since $\nu(Y_f) = \nu(\beta(\gamma_\iota )( Y_f ))$, we get
\begin{eqnarray}\label{Equation; Index}
[\Gamma \colon \Gamma_\mathrm{nor}] \nu(Y_f) = \sum_\iota \nu(\beta(\gamma_\iota )( Y_f )) = \nu(Y).
\end{eqnarray}
We note that if the measure of $Y$ is finite, the index of
$\Gamma_\mathrm{nor} \subset \Gamma$ is finite.
We regard $\Sigma_1 = G Y_f$ as a measurable embedding of $G$ into
$\Gamma_\mathrm{nor}$.
We note that $e$ is a fundamental domain for the $\Gamma_\mathrm{nor} / \Lambda_f$-action on $q (L^\infty \Sigma)^{H \times \Lambda} = (L^\infty \Sigma)^{H \times \Lambda}$.

The pair $(q f, H_f / H)$ is a fundamental pair for the $G$-action on $q (L^\infty \Sigma)^{H \times \Lambda}$.
Let $G_\mathrm{nor}$ be the normalizing subgroup of $H_f \subset G$.
By the same technique as above, we can find a $G_\mathrm{nor} \times \Gamma_\mathrm{nor}$-invariant projection $p$ in
$q (L^\infty \Sigma)^{H \times \Lambda}$ such that
$H_f / H$ acts on $p (L^\infty \Sigma)^{H \times \Lambda}$ trivially and that $q f$ gives a fundamental
domain for the $G_\mathrm{nor} / H_f$-action on $p (L^\infty \Sigma)^{H \times \Lambda}$.
Furthermore, since the measure of $X$ is finite,
the index $[G : G_\mathrm{nor}]$ is finite.

The projection $pe$ is a fundamental domain for the $\Gamma_\mathrm{nor} / \Lambda_f$-action
on $p (L^\infty \Sigma)^{H \times \Lambda}$ and
satisfies $\Tr_H(p e) < \infty$ by the equation (\ref{Equation; Radon--Nycodym derivative}). The projection $qf$ is a fundamental domain for the $G_\mathrm{nor} / H_f$-action
on $p (L^\infty \Sigma_1)^{H \times \Lambda}$.
Thus the measure space representing
$(p (L^\infty \Sigma)^{H \times \Lambda}, \Tr_H)$ gives a measurable embedding of $G_\mathrm{nor} / H_f$ into $\Gamma_\mathrm{nor} / \Lambda_f$.
Together with $G / H \sim_\mathrm{ME} G_\mathrm{nor} / H_f$ and $\Gamma_\mathrm{nor} / \Lambda_f \sim_\mathrm{ME} \Gamma_\mathrm{nor} / \Lambda \preceq_\mathrm{ME} \Gamma / \Lambda$, we get $G / H \preceq_\mathrm{ME} \Gamma / \Lambda$.

Suppose that $\Sigma$ is an ME coupling between $G$ and $\Gamma$.
Since $\mu(Y) < \infty$, the $G_\mathrm{nor} / H_f$ fundamental domain $q f \in p (L^\infty \Sigma_1)^{G \times \Lambda}$
satisfies $\Tr_H(q f) < \infty$.
We conclude that $p (L^\infty \Sigma_1)^{G \times \Lambda}$ gives an ME coupling between $G_\mathrm{nor} / H_f$ and $\Gamma_\mathrm{nor} / \Lambda_f$.
In addition, since the index $[\Gamma \colon \Gamma_\mathrm{nor}]$ is finite, we get $\Gamma_\mathrm{nor} / \Lambda_f \sim_\mathrm{ME} \Gamma / \Lambda$.
We conclude $G / H \sim_\mathrm{ME} \Gamma / \Lambda$.
\end{proof}

\subsection{Factorization up to ME}

We get factorization results on ME and measurable embedding.

\begin{theorem}\label{Theorem; Factorization of Product Groups}
Let $G = \prod_{i = 1} ^ m G_i$ be a product group of non-amenable groups $G_i$ and
let $\Gamma = \prod_{j = 1} ^ n \Gamma_j$ be a product group of
class $\mathcal{S}$ groups $\Gamma_j$.
Suppose $m \ge n$.
If $G \sim_\mathrm{ME} \Gamma$ $($resp.\ $G \preceq_\mathrm{ME} \Gamma$$)$,
then $m = n$ and the following hold:
\begin{enumerate}
\item
There exists $\sigma \in \mathfrak{S}_n$
so that $G_{\sigma(j)} \sim_\mathrm{ME} \Gamma_j$
$($resp. $G_{\sigma(j)} \preceq_\mathrm{ME} \Gamma_j$$)$;
\item\label{Assertion; Class S}
The group $\Gamma_j$ is non-amenable and $G_i \in \mathcal{S}$.
\end{enumerate}
\end{theorem}

The last claim is a consequence of the first and
Theorem 3.1 in \cite{Sako}.

\begin{theorem}\label{Theorem; Factorization with Amenable core}
Let $G_0$ and $\Gamma_0$ be amenable and
let $G_i \ (1 \le i \le m)$, $\Gamma_j \ (1 \le j \le n)$ be non-amenable groups in
the class $\mathcal{S}$.
Denote
$G = G_0 \times \prod_{i = 1} ^ m G_i$, $\Gamma = \Gamma_0 \times \prod_{j = 1} ^ n \Gamma_j$.
If $G \sim_\mathrm{ME} \Gamma$,
then $m = n$ and the following hold:
\begin{enumerate}
\item
There exists $\sigma \in \mathfrak{S}_n$
so that $G_{\sigma(j)} \sim_\mathrm{ME} \Gamma_j$;
\item
The group $\Gamma_0$ is finite, if and only if $G_0$ is finite.
\end{enumerate}
\end{theorem}

Until a middle point of the proof, both theorems require the same technique.
We proceed the proofs in the following assumptions.

\begin{framework}\label{Framework; Factorization of Direct Products}
Positive integers $m , n$ satisfy $m \ge n$.
A group $G_0$ is amenable
and groups $G_i \ (1 \le i \le m)$ are non-amenable groups.
A group $\Gamma_0$ is amenable and
and groups $\Gamma_j \ (1 \le j \le n)$ are in
the class $\mathcal{S}$.
We denote by $G$, $\Gamma$ the product groups
\begin{eqnarray*}
G = G_0 \times \prod_{i = 1} ^ m G_i, \quad
\Gamma = \Gamma_0 \times \prod_{j = 1} ^ n \Gamma_j.
\end{eqnarray*}
A measure space $(\Sigma, \nu)$ is an ergodic measurable embedding of $G$ into $\Gamma$.
We denote by $H_i$, $\Lambda_j$ the subgroups
\begin{eqnarray*}
    H_i = G_0 \times \prod_{k \neq i} G_k,\ 1 \le i \le m,
    \quad \Lambda_j = \Gamma_0 \times \prod_{l \neq j} \Gamma_l,\ 1 \le j \le n.
\end{eqnarray*}
We do not need ``$H_0$'', ``$\Lambda_0$''.

A measurable subset $X \subset \Sigma$ is
a $\Gamma$ fundamental domain and  a measurable subset $Y \subset \Sigma$ is a $G$ fundamental domain.
We denote by $\Tr_j = \Tr_{\Lambda_j}$
the trace on $(L^\infty \Sigma)^{\Lambda_j}$
defined as $\Tr_j(\cdot) = \Tr(\cdot \chi(\Gamma_j X))$.
We use the notations $\E^{(i)}_X$, $\E_X$, $\E^{(j)}_Y$ and $\E_Y$ for the function valued measures defined in Subsection $\ref{Subsection; Function Valued Measure}$:
\begin{eqnarray*}
    \E^{(i)}_Y = \E^{H_i}_Y
        \colon \widehat{(L^\infty \Sigma)^{H_i}_+} \longrightarrow \widehat{(L^\infty Y)_+}, &\ &
    \E _Y
        \colon \widehat{(L^\infty \Sigma)_+} \longrightarrow \widehat{(L^\infty Y)_+};\\
    \E^{(j)}_X = \E^{\Lambda_j}_X
        \colon \widehat{(L^\infty \Sigma)^{\Lambda_j}_+} \longrightarrow \widehat{(L^\infty X)_+}, &\ &
    \E _X
        \colon \widehat{(L^\infty \Sigma)_+} \longrightarrow \widehat{(L^\infty X)_+}.
\end{eqnarray*}
\end{framework}

The following proposition also proves \label{Assertion; Class S} in
Theorem \ref{Theorem; Factorization with Amenable core}.

\begin{proposition}\label{Proposition; Factorization of Direct Product Groups}
In Framework \ref{Framework; Factorization of Direct Products},
$m = n$ holds true.
If $\Gamma_0$ is finite, then $G_0$ is finite.
\end{proposition}

\begin{proof}
Proposition is proved by induction. We suppose $n = 1$.
The group $G = G_1 \times H_1 = G_1 \times (G_0 \times G_2 \times \cdots \times G_m)$
measurably embeds into $\Gamma = \Gamma_0 \times \Gamma_1$ by $\Sigma$.
The centralizing subgroup $Z_G(H_1)$ of $H_1$ is non-amenable.
Since $\Gamma_1$ is bi-exact relative to $\{\{1\}\}$, $\Gamma$ is bi-exact relative to $\{ \Gamma_0 \}$ (Lemma \ref{Lemma; Direct Product}).
There exists a partial embedding for $H_1 \preceq_\Sigma \Gamma_0$ in $\Sigma$,
by Theorem \ref{Theorem; Non embeddability implies amenability}. By remark \ref{Remark; MEm}, $H_1$ is amenable. It follows that
$H_1 = G_0$ and $m = 1$. If $\Gamma_0$ is finite, then $H_1 = G_0$ is also finite.

We suppose that the assertion holds true for a positive integer $n - 1$ and
that $G = G_m \times H_m$ measurably
embeds into $\Gamma$.
The group $\Gamma$ is bi-exact relative to
$\{\Lambda_i \ | \ 1 \le i \le n \}$, since $\Gamma_i$ is bi-exact relative to
$\{1\}$ (Lemma \ref{Lemma; Direct Product}).
The centralizing subgroup of $H_m$ in $G$ is non-amenable as $G_m$ is not amenable.
By Theorem \ref{Theorem; Non embeddability implies amenability},
we have a measurable embedding
$H_m\preceq_\mathrm{ME} \Lambda_j$.
By the induction hypothesis, we get $m - 1 \le n -1$. It also follows that if $m = n$ (equivalently $m - 1 = n - 1$) and
if $\Gamma_0$ is finite, then $G_0$ is also finite.
\end{proof}

For $1 \le j \le n$, there exists $1 \le \sigma(j) \le n = m$ satisfying
$H_{\sigma(j)} \preceq_{\Sigma} \Lambda_j$ by Theorem \ref{Theorem; Non embeddability implies amenability}.
We claim that $\sigma$ defines a map.

\begin{lemma}\label{Lemma; definition of sigma}
In Framework \ref{Framework; Factorization of Direct Products},
if $H_i \preceq_{\Sigma} \Lambda_j$ and $H_k \preceq_{\Sigma} \Lambda_j$, then $i = k$.
\end{lemma}

\begin{proof}
By the assumptions, there exist
projections $e_i, e_k$ in $(L^\infty \Sigma)^{\Lambda_j}$
satisfying $h_i (e_i) = e_i$ $(h_i \in H_i)$, $h_k (e_k) = e_k$ $(h_k \in H_k)$ and
\begin{eqnarray*}
  0 < \Tr_j(e_i) < \infty, \quad 0 < \Tr_j(e_k) < \infty.
\end{eqnarray*}
For any $g \in G_i$ and $\gamma \in \Gamma_j$, the projection $g \gamma (e_i)$ is also
invariant under the action of $H_i$ and $\Lambda_j$ and the trace $\Tr_j (g \gamma (e_i))$ is equal to
$\Tr_j(e_i)$. Since the action of $G_i \times \Gamma_j$ on $(L^\infty \Sigma)^{H_i \times \Lambda_j}$
is ergodic,
the projection $\bigvee \{ g \gamma (e_i) \ | \ g \in G_i, \gamma \in \Gamma_j\}$ is $1$.
Thus there exists a projection $\hat{e}_i$ obtained by a finite union of
$\{ g \gamma (e_i)\}$ such that $h_i (\hat{e}_i) = \hat{e}_i$ $(h_i \in H_i)$ and
\begin{eqnarray*}
    0 < \Tr_j(\hat{e}_i) < \infty, \quad
    \Tr_j(e_k) / 2 < \Tr_j(e_k \hat{e}_i).
\end{eqnarray*}

Assume $i \neq k$. Denote by $\overline{\mathcal{C}}$ the convex norm closure of
$\{ g(\hat{e}_i) \ | \ g \in G_i\}$ in $L^2((L^\infty \Sigma)^{\Lambda_j}, \Tr_j)$.
The element $\xi \in \overline{\mathcal{C}}$ having the minimal value of $2$-norm
is fixed under $G_i$ as well as $H_i$.
Since we have $g(e_k) = e_k$ for $g \in G_i \subset H_k$, the following inequality holds true:
\begin{eqnarray*}
    \langle e_k, g(\hat{e}_i) \rangle = \Tr_j(e_k g(\hat{e}_i))
    = \Tr_j(g(e_k \hat{e}_i)) = \Tr_j(e_k \hat{e}_i) > \Tr_j(\hat{e}_i) / 2.
\end{eqnarray*}
The vector $\xi$ satisfies $\langle e_k, \xi \rangle \ge \Tr_j(\hat{e}_i) / 2$
and thus $\xi$ is not zero. Since $\xi$ is fixed under $G = G_i \times H_i$,
a non-trivial level set $\Omega \in \Sigma$ of $\xi$
is also fixed under $G$.
The measure of an $H_j$ fundamental domain is $\Tr_j(\Omega) < \infty$.
The measurable subset $\Omega$ gives a measurable embedding
of $G$ into $H_j$, which contradicts Proposition \ref{Proposition; Factorization of Direct Product Groups}.
 We conclude $i = k$
\end{proof}

We prove that $\sigma$ defines an injective map. By $m = n$, $\sigma$ is also surjective.

\begin{lemma}\label{Lemma; Injectivity of sigma}
In Framework \ref{Framework; Factorization of Direct Products},
if $H_i \preceq_{\Sigma} \Lambda_j$ and $H_i \preceq_{\Sigma} \Lambda_l$, then $j = l$.
\end{lemma}

\begin{proof}
There exist
projections $f_j, f_l \in (L^\infty \Sigma)^{H_i}$
satisfying $\lambda_j (f_j) = f_j$ $(\lambda_j \in \Lambda_j)$, $\lambda_l (f_l) = f_l$ $(\lambda_l \in \Lambda_l)$ and
\begin{eqnarray*}
  0 < \Tr_j(f_j) < \infty, \quad 0 < \Tr_j(f_l) < \infty.
\end{eqnarray*}
Since $\Sigma$ is an ergodic measurable embedding, the projection $\bigvee \{ g \gamma (f_j) \ | \ g \in G_i, \gamma \in \Gamma_j\}$ is $1$.
Replacing $f_j$ with a bigger projection, we may assume that $f_j f_l$ is not zero.

Assuming $j \neq l$, we deduce a contradiction.
Denote $\Delta = \Gamma_0 \times \prod_{k \neq j, l} \Gamma_k = \Lambda_j \cap \Lambda_k$.
The function valued measures $\E^{(j)}_X$ and $\E^{(l)}_X$
satisfy the following:
\begin{eqnarray*}
        \mathfrak{E}^\Delta_X (f_j f_l) (x)
    &=& \sum_{\gamma_j \gamma_l \in \Gamma_j \times \Gamma_i}
        f_j f_l (\gamma_j \gamma_l x)\\
    &=& \sum_{\gamma_j \gamma_l \in \Gamma_j \times \Gamma_i}
        f_j(\gamma_j x)  f_l(\gamma_l x)\\
    &=& \sum_{\gamma_j \in \Gamma_j} f_j(\gamma_j x)
        \sum_{\gamma_l \in \Gamma_l} f_l(\gamma_l x)\\
    &=& \E^{(j)}_X (f_j) (x) \E^{(l)}_X (f_l) (x),  \quad \mathrm{\ a.e.\ } x \in X.
\end{eqnarray*}
The projection $f_j f_l$ is $H_i$-invariant.
The value of the measurable function $\E^\Delta_X (f_j f_l)$
is finite almost everywhere,
since the functions $\E^{(j)}_X (f_j)$ and $\E^{(l)}_X (f_l)$ are integrable.
It follows that $H_i \preceq_\Sigma \Delta$, by
Lemma \ref{Lemma; FVM and MEm}.
This contradicts Proposition
\ref{Proposition; Factorization of Direct Product Groups}.
\end{proof}

\begin{proof}
[Proof for the assertion 1
in Theorem \ref{Theorem; Factorization of Product Groups}]
Let $G_0$ and $\Gamma_0$ be trivial groups in Framework
\ref{Framework; Factorization of Direct Products}.
By redefining the indices,
we may assume $H_i \preceq_\Sigma \Lambda_i$.

We take a projection $e_i \in (L^\infty \Sigma)^{H_i \times \Lambda_i}$ satisfying
$0 < \Tr_i(e_i) < \infty$. We may assume that $\mathfrak{E}^{(i)}_X (e_i)$ is bounded.
By replacing $e_i$ with a finite union of projections
$g \gamma (e_i) \ (g \in G_i, \gamma \in \Gamma_i)$,
we may also assume that
the product $e = \prod_{i = 1}^n e_i$ is not zero. By direct computations, we get the following equation:
\begin{eqnarray*}
    \mathfrak{E}_Y(e)(y)
    &=& \sum_{g \in G} e(g y)
     =  \sum_{(g_1, g_2, \cdots, g_n) \in G} \prod_{i = 1}^n e_i(g_i y)\\
    &=& \prod_{i = 1}^n \sum_{g_i \in G_i} e_i(g_i y)
     =  \prod_{i = 1}^n \mathfrak{E}^{(i)}_Y (e_i) (y), \quad \mathrm{\ a.e.\ } y \in Y.
\end{eqnarray*}
We also get $\E_X(e) = \prod_{i = 1}^n \E^{(i)}_X (e_i)$.
Then it turns out that $\mathfrak{E}_Y(e)$ is integrable, since
\begin{eqnarray*}
        \int_Y \mathfrak{E}_Y(e) d \nu
    &=& \int_\Sigma e d \nu
     =  \int_X \mathfrak{E}_X(e) d \nu\\
    &=& \int_X \prod_{i = 1}^n \mathfrak{E}^{(i)}_X (e_i) d \nu
    \le \nu(X) \prod_{i = 1}^n \sup_x \mathfrak{E}^{(i)}_X (e_i)(x) < \infty.
\end{eqnarray*}
On the support $W \subset Y$ of $\E_Y (e)$, the function $\E^{(i)}_Y (e_i)$ satisfies
\begin{eqnarray*}
        \E^{(i)}_Y (e_i)(y)
    \le \E^{(i)}_Y (e_i)(y)
	    \times \prod_{j \neq i} \E^{(j)}_Y (e_j)(y)
     = \E_Y (e)(y), \quad {\rm \ a.e. \ } y \in W,
\end{eqnarray*}
since $\E^{(j)}_Y (e_j)$ is $(\{0, 1, \cdots,\infty\})$-valued on $W$. It follows that the function $\mathfrak{E}^{(i)}_Y (e_i)$ is integrable on $W$.
Since $\E_Y (e)$ is not zero, $\E^{(i)}_Y (e_i)$ is also not zero on $W$.
By Proposition \ref{Proposition; ME and MEm passes to Quotient Groups} for quotients $G_i \cong G / H_i$ and $\Gamma_i \cong \Gamma / \Lambda_i$,
we get the conclusion in the two cases $\nu(Y) < \infty$ and $\nu(Y) = \infty$.
\end{proof}

\begin{proof}
[Proof for the assertion 1
in Theorem \ref{Theorem; Factorization with Amenable core}]
Let $G_0$, $\Gamma_0$ be amenable groups and let $G_i, \Gamma_i \ (1 \le i \le n)$
be non-amenable groups in the class $\mathcal{S}$.
By replacing the indices, we may assume that
$H_i \preceq_\Sigma \Lambda_i$ for any $i$.
By replacing the roles on $G$ and $\Gamma$, there exists $\rho \in \mathfrak{S}_n$
such that
$\Lambda_i \preceq_\Sigma H_{\rho(i)}$.
By Proposition \ref{Proposition; ME and MEm passes to Quotient Groups},
we have only to show that $\rho(i) = i$.

Assume that $k = \rho(i) \neq i$. Since $H_i \preceq_\Sigma \Lambda_i$,
there exists a projection
$e \in (L^\infty \Sigma)^{H_i \times \Lambda_i}$ with $0 < \Tr_i (e) < \infty$.
Since $\Lambda_i \preceq_\Sigma H_k$, by Lemma \ref{Lemma; Find a Fundamental Pair}, there exist
a projection $f \in (L^\infty \Sigma)^{H_k \times \Lambda_i}$ and
a finite subgroup $G_{k, f} \subset G_k$ so that
the pair $(f, G_{k, f})$ is a fundamental pair for the $G_k$-action on
$(L^\infty \Sigma)^{H_k \times \Lambda_i}$. Let $\{g_\iota\}_{\iota \in I}$ be a set of
representatives for the left cosets $G_k / G_{k, f}$.
The projections $\{ g_\iota(f) \}_{\iota \in I}$
gives a partition of $1$. Since the $G_k$-action preserves $\Tr_i = \Tr_{\Lambda_i}$
and fixes $e$, we get
\begin{eqnarray*}
	     \Tr_i (e)
	  =  \sum_{\iota \in I} \Tr_i(e g_\iota(f))
	  =  \sum_{\iota \in I} \Tr_i(g_\iota(e f))
	  =  |I| \Tr_i(e f).	
\end{eqnarray*}
This contradicts $0 < \Tr_i (e) < \infty$ and $|I| = \infty$. Therefore we get $k = i$.
\end{proof}

\subsection{Separately Ergodic Couplings}

\begin{definition}
For a measure preserving group action of
$G = G_0 \times \prod_{i = 1} ^ n G_i$ on a standard probability space $X$,
we say that the action is \textbf{separately ergodic} when the subgroups
$H_i = G_0 \times \prod_{k \neq i} G_k \ (1 \le i \le n)$ act on $X$ ergodically.
For a measurable embedding $\Sigma$ of the product group $G$ and arbitrary countable group
$\Gamma$, we say that the action is \textbf{separately ergodic} when
the groups $H_i \times \Gamma$ act on $\Sigma$ ergodically.
\end{definition}

For a separately ergodic couplings, we get a stronger conclusion than the previous subsection.

\begin{theorem}\label{Theorem; Factorization by Sep Erg Coupling}
Let $G$ and $\Gamma$ be product groups which satisfy the assumptions in Theorem
\ref{Theorem; Factorization of Product Groups}.
Let $\Sigma$ be a measurable embedding of $G$ into $\Gamma$.
If $\Sigma$ is separately ergodic, then $m = n$ and there exist $\sigma \in \mathfrak{S}_n$ and subgroups $G_{i, \mathrm{fin}} \subset G$, $\Gamma_{i, \mathrm{fin}} \subset \Gamma_{i, \mathrm{nor}} \subset \Gamma_i$ $(1 \le i \le n)$ with the following properties:
\begin{enumerate}
\item
The subgroup $G_{i, \mathrm{fin}} \subset G_i$ is finite and normal.
The subgroup $\Gamma_{i, \mathrm{fin}}$ is finite and $\Gamma_{i, \mathrm{nor}}$ normalizes $\Gamma_{i, \mathrm{fin}}$;
\item
The group $G_{\sigma(i)} / G_{\sigma(i), \mathrm{fin}}$ is isomorphic
to $\Gamma_{i, \mathrm{nor}} / \Gamma_{i, \mathrm{fin}}$.
\item
The coupling index of $\Sigma$ satisfies
\begin{eqnarray*}
[\Gamma \colon G]_\Sigma = \prod_{i = 1}^n
\frac{|\Gamma_{i, \mathrm{fin}}| [ \Gamma_i \colon \Gamma_{i, \mathrm{nor}}] }
{ | G_{\sigma(i), \mathrm{fin}} |}.
\end{eqnarray*}
If $\Sigma$ is an ME coupling, then $[ \Gamma_i \colon \Gamma_{i, \mathrm{nor}}] < \infty$ and
$G_{\sigma(i)}$ and $\Gamma_i$ are commensurable up to finite kernel.
\end{enumerate}
\end{theorem}

\begin{theorem}\label{Theorem; Factorization by Sep Erg Coupling with Ame Core}
Let $G$ and $\Gamma$ be product groups which satisfy the assumptions in Theorem
\ref{Theorem; Factorization with Amenable core}. If there exists a separately ergodic
ME coupling between $G$ and $\Gamma$,
then $m=n$ and there exists $\sigma \in \mathfrak{S}_n$
so that $G_{\sigma(i)}$ and $\Gamma_i$ are commensurable up to finite kernel.
\end{theorem}

We proceed the proof for the two theorems in Framework
\ref{Framework; Factorization of Direct Products}.

\begin{proof}
Suppose that the
measurable embedding $\Sigma$ is separately ergodic. By the previous subsection,
$m =n$ and there exists $\sigma \in \mathfrak{S}_n$ satisfying
$H_{\sigma(i)} \preceq_\Sigma \Lambda_i$.
For simplicity of notations, we change the indices on $G_i$ so that
$H_i \preceq_\Sigma \Lambda_i$.

Let a pair $(e_i, \Gamma_{i, \mathrm{fin}})$ of a projection
$e_i \in (L^\infty \Sigma)^{H_i \times \Lambda_i}$ and a finite subgroup
$\Gamma_{i, \mathrm{fin}} \subset \Gamma_i$ be a fundamental pair for the $\Gamma_i$-action
on $(L^\infty \Sigma)^{H_i \times \Lambda_i}$ (Lemma \ref{Lemma; Find a Fundamental Pair}).
Since $e_i \perp \gamma(e_i)$ for $\gamma \in \Gamma_i \cap
(\Gamma_{i, \mathrm{fin}})^c$ and
the group $\Gamma_{i, \mathrm{fin}}$ acts on $e_i (L^\infty \Sigma)^{H_i \times \Lambda_i}$
trivially, every projection $e_i^\prime$ in
$e_i (L^\infty \Sigma)^{H_i \times \Lambda_i}$ satisfy
\begin{eqnarray*}
    e_i^\prime = e_i e_i^\prime
    = \sum_{\gamma \Gamma_{i, \mathrm{fin}} \in \Gamma_i / \Gamma_{i, \mathrm{fin}}} e_i \gamma (e_i^\prime)
    = e_i \bigvee_{\gamma \in \Gamma_i} \gamma (e_i^\prime)
    = e_i \bigvee_{\gamma \in \Gamma} \gamma (e_i^\prime).
\end{eqnarray*}
Letting $X^\prime \subset X$ be the support of
$\E^{(i)}_X (e_i^\prime)$, the projection $e^\prime_i$ is of the form
$e_i \chi(\Gamma X^\prime)$. The measurable subset
$X^\prime \subset X \cong \Gamma \backslash \Sigma$
is $H_i$-invariant since $\E^{(i)}_X$ is $G$-equivariant. Since the embedding $\Sigma$ is separately ergodic,
it must be null or co-null. We get $e_i^\prime = e_i$ or $0$.
This means that $e_i$ is a minimal projection in $(L^\infty \Sigma)^{H_i \times \Lambda_i}$.

Let $\mathcal{P}_i$ be the set of minimal projections in $(L^\infty \Sigma)^{H_i \times \Lambda_i}$.
The $G_i$-action and $\Gamma_i$-action on $\mathcal{P}_i$ commute with each other.
Since $(e_i, \Gamma_{i, \mathrm{fin}})$ is a fundamental pair,
the action of $\Gamma_i$ on $\mathcal{P}_i$ is
transitive. The stabilizer of $e_i$ is $\Gamma_{i, \mathrm{fin}}$.
Let $G_{i, \mathrm{fin}} \subset G_i$ be the stabilizer of $e_i$,
and $\Gamma_{i, \mathrm{nor}} \subset \Gamma_i$
be the collection of elements $\gamma \in \Gamma_i$ for which there exists $g \in G_i$
satisfying $\gamma (e_i) = g^{-1} (e_i)$.
If $g \in G_i$ and $\gamma \in \Gamma_{i, \mathrm{nor}}$ satisfy this relation,
then for $g_f \in G_{i, \mathrm{fin}}$ and $\gamma_f \in \Gamma_{i, \mathrm{fin}}$ we get
\begin{eqnarray*}
    g^{-1} g_f g (e_i)
    = g^{-1} g_f \gamma^{-1} (e_i)
    = g^{-1} \gamma^{-1} g_f (e_i)
    = g^{-1} \gamma^{-1} (e_i)
    = e_i,\\
    \gamma^{-1} \gamma_f \gamma (e_i)
    = \gamma^{-1} \gamma_f g^{-1} (e_i)
    = \gamma^{-1} g^{-1} \gamma_f (e_i)
    = \gamma^{-1} g^{-1} (e_i)
    = e_i.
\end{eqnarray*}
It turns out that $G_i$, $\Gamma_{i, \mathrm{nor}}$ normalize $G_{i, \mathrm{fin}}$, $\Gamma_{i, \mathrm{fin}}$ respectively. If
$g_a, g_b \in G_i$ and $\gamma_a, \gamma_b \in \Gamma_{i, \mathrm{nor}}$ satisfy relations
$\gamma_a (e_i) = g_a^{-1} (e_i)$, $\gamma_b (e_i) = g_b^{-1} (e_i)$, then
$\gamma_a^{-1} (e_i) = g_a (e_i)$ and
\begin{eqnarray*}
    \gamma_a \gamma_b (e_i)
    = \gamma_a g_b^{-1} (e_i)
    = g_b^{-1} \gamma_a (e_i)
    = g_b^{-1} g_a^{-1} (e_i)
    = (g_a g_b)^{-1} (e_i).
\end{eqnarray*}
It follows that $\Gamma_{i, \mathrm{nor}}$ is a subgroup of $\Gamma_i$ and that when we define
a map
\begin{eqnarray*}
    \phi_i : G_i / G_{i, \mathrm{fin}} \ni g G_{i, \mathrm{fin}}
                \mapsto \gamma \Gamma_{i, \mathrm{fin}} \in \Gamma_{i, \mathrm{nor}} / \Gamma_{i, \mathrm{fin}}
\end{eqnarray*}
by $\gamma (e_i) = g^{-1} (e_i)$, this gives a group isomorphism.

We next claim that the function valued measures satisfy
\begin{eqnarray*}
    \E^{(i)}_X(e_i) = | \Gamma_{i, \mathrm{fin}} | 1_X, \quad
    \E^{(i)}_Y(e_i) = | G_{i, \mathrm{fin}} | 1_{Y_i},
\end{eqnarray*}
where $Y_i$ is the support of $\E^{(i)}_Y(e_i)$. Define projections $q_i, q \in L^\infty Y$ by
$q = 1_{Y_i}$ and $q = \prod_{i = 1}^n q_i$. The measurable subset $Y_0 = \bigcap_{i=1}^n Y_i$ corresponds to $q$.
Take a $\Gamma$ fundamental domain $X_i \subset \Sigma$ as $\chi(X_i) \le e_i$. The measurable set corresponding to $e_i$ can be written as $\Gamma_{i, \mathrm{fin}} \Lambda_i X_i$, since $\gamma (e_i) = e_i \ (\gamma \in \Gamma_{i, \mathrm{fin}})$,
$\gamma (e_i) \perp e_i \ (\gamma \in \Gamma_i \cap \Gamma_{i, \mathrm{fin}}^c)$ and $e_i$ is $\Lambda_i$-invariant. The function valued measure satisfies $\E^{\Lambda_i}_{X_i}(e_i) = | \Gamma_{i, \mathrm{fin}} | 1_{X_i}$ and this confirms the first equation by the identification $X_i \cong \Gamma \backslash \Sigma \cong X$. The proof for the second equation is the same.
Define $e = \prod_{i = 1}^n e_i \in L^\infty \Sigma$.
The function valued measures of $e$ with respect to $\Gamma_0$ and $G_0$ are
\begin{eqnarray}\label{Equation; FVM and Minimal Proj}
\qquad
    \E^{\Gamma_0}_X(e) &=& \prod_{i = 1}^n \E^{(i)}_X (e_i)
    = \prod_{i = 1}^n |\Gamma_{i, \mathrm{fin}}| 1_X,
     \label{Equation; FVM and Minimal Proj X}\\
    \mathfrak{E}^{G_0}_Y (e)
    &=& \prod_{i = 1}^n \mathfrak{E}^{(i)}_Y (e_i)
    = \prod_{i = 1}^n |G_{i, \mathrm{fin}}| q.
     \label{Equation; FVM and Minimal Proj Y}
\end{eqnarray}

Define subgroups $G_\mathrm{fin} \subset G$
and $\Gamma_\mathrm{fin} \subset \Gamma_\mathrm{nor} \subset \Gamma$ by
\begin{eqnarray*}
    G_\mathrm{fin} = G_0 \times \prod_{i = 1}^n G_{i, \mathrm{fin}}, \quad
    \Gamma_\mathrm{fin} = \Gamma_0 \times \prod_{i = 1}^n \Gamma_{i, \mathrm{fin}}, \quad
    \Gamma_\mathrm{nor} = \Gamma_0 \times \prod_{i = 1}^n \Gamma_{i, \mathrm{nor}}.
\end{eqnarray*}
We next claim
\begin{eqnarray}\label{Equation; Index in Sep Erg case}
    \nu (Y) = [ \Gamma : \Gamma_\mathrm{nor} ] \nu (Y_0).
\end{eqnarray}
We denote by  $q_i$ the union of $G_i$-orbit of $e_i$,
\begin{eqnarray*}
    q_i
	= \bigvee_{g \in G_i} g(e_i)
	= \bigvee_{\gamma \in \Gamma_{i, \mathrm{nor}}} \gamma(e_i).
\end{eqnarray*}
The measurable subset $Y_i \subset Y \cong G \backslash \Sigma$ corresponds
to $q_i$. We note that for $\gamma, \gamma^\prime \in \Gamma_i$, we get either
$\gamma(q_i) = \gamma^\prime(q_i)$
$(\gamma^{-1} \gamma^\prime \in \Gamma_{i,\mathrm{nor}})$ or
$\gamma(q_i) \perp \gamma^\prime(q_i)$
$(\gamma^{-1} \gamma^\prime \in
\Gamma_i \cap (\Gamma_{i,\mathrm{nor}})^c)$.
Then for $\gamma, \gamma^\prime \in \Gamma$, we get
\begin{eqnarray*}
    \gamma(q) = \gamma^\prime(q),
    \ \gamma^{-1} \gamma^\prime \in \Gamma_\mathrm{nor}, \quad
    \gamma(q) \perp \gamma^\prime(q),
    \ \gamma^{-1} \gamma^\prime \in \Gamma \cap (\Gamma_\mathrm{nor})^c.
\end{eqnarray*}
It follows that representatives $\{\gamma_\iota\}_\iota$ for $\Gamma / \Gamma_\mathrm{nor}$ give
a partition $\{\gamma_\iota q\}_\iota$ of $1_\Sigma$. Since the measurable sets
$\{\gamma_\iota Y_0 \}_\iota$ have
the same measure, we have the equation (\ref{Equation; Index in Sep Erg case}).

Suppose $G_0 = \Gamma_0 = \{1\}$. Since $|\Gamma_{i, \mathrm{fin}}|$ is finite, by (\ref{Equation; FVM and Minimal Proj X}) and (\ref{Equation; FVM and Minimal Proj Y}), we get
\begin{eqnarray*}
    |G_\mathrm{fin}| \nu(Y_0)
    = \int_Y \mathfrak{E}_Y (e) \nu
    = \int_\Sigma e d \nu
    = \int_X \mathfrak{E}_X (e) \nu
    = |\Gamma_\mathrm{fin}| \nu(X) < \infty.
\end{eqnarray*}
It follows that the subgroups $G_{i, \mathrm{fin}}$ are finite. Furthermore, the coupling index
of $\Sigma$ is given by
\begin{eqnarray*}
    [ \Gamma : G ] _{\Sigma}
    = \nu(Y) / \nu(X)
    = [\Gamma : \Gamma_\mathrm{nor} ] \nu(Y_0) / \nu(X)
    = [\Gamma : \Gamma_\mathrm{nor}] |\Gamma_\mathrm{fin}| / |G_\mathrm{fin}|.
\end{eqnarray*}
In particular, if $[ \Gamma : G ] _{\Sigma} < \infty$, then
$[\Gamma : \Gamma_\mathrm{nor}] < \infty$.
The map
$\phi_i$ gives an isomorphism between $G_i / G_{i, \mathrm{fin}}$ and $\Gamma_{i, \mathrm{nor}} / \Gamma_{i, \mathrm{fin}}$.
Theorem \ref{Theorem; Factorization by Sep Erg Coupling} was confirmed.

In turn, we suppose that $\nu(Y) < \infty$ and $G, \Gamma$ are product groups satisfying the assumptions
in Theorem \ref{Theorem; Factorization with Amenable core}.
The proof of
Theorem \ref{Theorem; Factorization with Amenable core} has shown that $\Lambda_i \preceq_{\Sigma} H_i$ and that $\Tr_{H_i}$ is a scalar multiple of $\Tr_{\Lambda_i}$.
Thus the projection $e_i$ satisfies
$0< \Tr_{H_i} (e_i) = \int_Y \mathfrak{E}^{(i)}_Y (e_i) d \nu < \infty$.
The group $G_{i, \mathrm{fin}}$ is finite as $\mathfrak{E}^{(i)}_Y (e_i) = |G_{i, \mathrm{fin}}| 1_{Y_i}$ is integrable.
The index $[ \Gamma \colon \Gamma_\mathrm{nor} ] = \nu(Y) / \nu(Y_0)$
is also finite by the equation (\ref{Equation; Index in Sep Erg case}). It follows that $\Gamma_{i, \mathrm{nor}}$ is a finite index subgroup
of $\Gamma_i$. The map
$\phi_i$ gives an isomorphism between $G_i / G_{i, \mathrm{fin}}$ and $\Gamma_{i, \mathrm{nor}} / \Gamma_{i, \mathrm{fin}}$.
This confirms Theorem
\ref{Theorem; Factorization by Sep Erg Coupling with Ame Core}.
\end{proof}

\subsection{OE Strong Rigidity Theorems}
\label{Subsection; OE Str Rigidity Thms}
\begin{definition}
Let $G$ and $\Gamma$ be arbitrary countable groups. Suppose that
$\alpha$ is a free e.m.p.\ action of $G$ on a standard probability space $X$
and that $\phi : G \rightarrow \Gamma$ be a group homomorphism with finite kernel.
Consider the $G \times \Gamma$-action $\mathfrak{A}$ defined on $\Sigma = \Gamma \times X$ by
\begin{eqnarray*}
     \mathfrak{A}(\gamma_0, g) (\gamma, x)
     = (\gamma_0 \gamma \phi(g)^{-1}, \alpha_g(x))
\end{eqnarray*}
and choose a fundamental domain $Y$ for the $G$ action.
\textbf{The induced action} $\mathrm{Ind}_G^\Gamma(\alpha, \phi)$ is a $\Gamma$-action
on $Y \cong G \backslash \Sigma$ defined by
\begin{eqnarray*}
     \gamma_0 (\mathfrak{A}(G) (\gamma, x))
     = \mathfrak{A}(G) (\gamma_0 \gamma, x).
\end{eqnarray*}
\end{definition}

The induced action is free, ergodic and measure preserving.
If the group homomorphism $\phi$ is an isomorphism, then
the induced action $\beta = \mathrm{Ind}_G^\Gamma(\alpha, \phi)$
is conjugate to the action $\alpha$. The measure of $Y$ is finite if and only if the image of $\phi$ is a finite index subgroup of $\Gamma$.

\begin{definition}
A group $\Gamma$ is said to be in $\mathcal{S}_0$ when $\Gamma \in \mathcal{S}$ and does not have
a pair of subgroups $\{1\} \neq \Gamma_\mathrm{fin} \subset \Gamma_\mathrm{nor} \subset \Gamma$
satisfying
$(1)$ $\Gamma_\mathrm{fin}$ is finite,
$(2)$ $\Gamma_\mathrm{nor}$ normalizes $\Gamma_\mathrm{fin}$,
$(3)$ $\Gamma_\mathrm{nor} \subset \Gamma$ is finite index.
\end{definition}

ICC groups satisfy these three conditions.

\begin{theorem}\label{Theorem; OE Str Rigidity for Product Group}
Let $G = \prod_{i = 1} ^n G_i$ be a product group of non-amenable groups and
let $\Gamma = \prod_{i = 1} ^n \Gamma_i$ be a product group of groups in $\mathcal{S}_0$.
Suppose that $\alpha$ is a free e.m.p.\ $G$-action on a standard probability space $X$
and that $\beta$ is a free e.m.p.\ $\Gamma$-action on a standard measure space $Y$.
If the two group actions $\alpha$ and $\beta$ are SOE
with compression constant $s \in (0, \infty]$,
$($that is, $\R_\alpha^s \cong \R_\beta$$)$, and if $\alpha$ is separately ergodic,
then there exist $\sigma \in \mathfrak{S}_n$ and
a group homomorphism from $\phi_i \colon G_{\sigma(i)} \rightarrow \Gamma_i$ with the following properties;
\begin{enumerate}
\item
The $\Gamma$-action $\beta$ is conjugate to the induced action $\mathrm{Ind}_G^\Gamma(\alpha, \phi)$, where
$\phi$ is the group homomorphism from $G$ to $\Gamma$ given by
$\phi((g_i)_{\sigma(i)}) = (\phi_i(g_i))$.
\item
The compression constant $s$ satisfies
\begin{eqnarray*}
s = \prod_{i = 1}^n
\frac{[ \Gamma_i \colon \mathrm{image}(\phi_i)] }{ | \ker(\phi_i) |}.
\end{eqnarray*}
If $s < \infty$, then $[ \Gamma_i \colon \mathrm{image}(\phi_i)] < \infty$ and $G_{\sigma(i)}, \Gamma_i$ are commensurable up to finite kernel.
\end{enumerate}
\end{theorem}

\begin{proof}
Let $\R$ be a type $\mathrm{II}$ relation on a standard measure space $(Z, \nu)$, which gives
SOE between $\alpha$ and $\beta$. Namely, $X, Y \subset Z$ be measurable subsets with
$\mu(X) = 1, \mu(Y) = s$ and that $\R_\alpha = \R \cap (X \times X)$,
$\R_\beta = \R \cap (Y \times Y)$. We consider that the measure space $\Sigma = \R \cap (X \times Y)$
as a measurable embedding of $G$ into $\Gamma$ and that $X$ is a separately ergodic $G$-space. Then the embedding $\Sigma$ is separately ergodic.

We use notations in Framework \ref{Framework; Factorization of Direct Products}.
For the simplicity for notations, we assume that $H_i \preceq_\Sigma \Lambda_i$.
Let $e_i \in (L^\infty \Sigma)^{H_i \times \Lambda_i}$ be a minimal projection and let
$G_{i, \mathrm{fin}} \subset G_i$ and  $\Gamma_{i, \mathrm{fin}} \subset \Gamma_{i, \mathrm{nor}} \subset \Gamma_i$ be
subgroups given in the previous proof. The subgroup $\Gamma_{i, \mathrm{fin}}$
is finite and the inclusion $\Gamma_{i, \mathrm{nor}} \subset \Gamma_i$ has finite index.
The condition $\Gamma \in \mathcal{S}_0$ means $\Gamma_{i, \mathrm{fin}} = \{1\}$.
We get a surjective group homomorphism $\phi_i
\colon G_i \rightarrow \Gamma_{i, \mathrm{nor}}$ with kernel $G_{i, \mathrm{fin}}$ by $\phi(g) e_i = g^{-1} e_i$.
Defining $\phi
\colon G \rightarrow \Gamma$ by $\phi((g_i)) = (\phi_i(\gamma_i))$,
we get $\phi(g) e = g^{-1} e$ for $g \in G$.
By using (\ref{Equation; FVM and Minimal Proj X}), the projection $e = \prod_{i = 1}^n e_i$ satisfies
$\E_X(e) = \prod_{i =1}^n |\Gamma_{i, \mathrm{fin}}| 1_X = 1_X$.
We identify the measurable set $X$ and the support of $e$. We also identify the
measurable set $\Sigma$ and $\Gamma \times X$ by
$\Gamma \times X \ni (\gamma, x) \mapsto \gamma (x) \in \Sigma$.
The $G$-action on $L^\infty \Sigma$ satisfies
\begin{eqnarray*}
    g (f \gamma(e)) = \alpha_g (f) \gamma g (e)
    = \alpha_g (f) \gamma \phi(g)^{-1}(e),
    \quad f \in L^\infty X \cong (L^\infty \Sigma)^{\Gamma}.
\end{eqnarray*}
It follows that the $G$-action on $\Gamma \times X$ can be written as
\begin{eqnarray*}
    g (\gamma, x) = (\gamma \phi(g)^{-1}, \alpha_g(x)), \quad
    \gamma \in \Gamma, {\rm \ a.e.\ } x \in X.
\end{eqnarray*}
Since $Y$ is isomorphic to $G \backslash \Sigma$ as a $\Gamma$-space,
the $\Gamma$-action $\beta$ is isomorphic to the action $\mathrm{Ind}_G^\Gamma(\alpha, \phi)$.
\end{proof}

\begin{corollary}
\label{Corollary; Trivial Fundamental Group}
Let $G = \prod_{i = 1} ^n G_i$ and $\Gamma = \prod_{i = 1} ^n \Gamma_i$
be product groups of non-amenable groups in $\mathcal{S}_0$.
Suppose that $\alpha$ is a free e.m.p.\ $G$-action on a standard probability space $X$
and that $\beta$ is a free e.m.p.\ $\Gamma$-action on a standard finite measure space $Y$.
If the two group actions $\alpha$ and $\beta$ are stably orbit equivalent with constant $s \in (0, \infty)$, that is $\R_\alpha^s \cong \R_\beta$, and if $\alpha$ and $\beta$ are separately ergodic,
then $s = 1$.
In particular, the fundamental group $\mathcal{F}(\R_\alpha)$ is $\{1\}$.
\end{corollary}

\begin{proof}
Since the group $G_i$ is in $\mathcal{S}_0$, $G_i$ has no normal finite subgroup
other than $\{1\}$.
Thus we get
$s = \prod_{i = 1}^n [ \Gamma_i \colon \Gamma_{i, \mathrm{nor}}] \ge 1$ by Theorem.
By replacing the roles on $G$ and $\Gamma$, we also get $s^{-1} \ge 1$.
\end{proof}

\begin{corollary}
\label{Corollary; Trivial Out Group}
Let $G$ and $\Gamma$
be as in Corollary \ref{Corollary; Trivial Fundamental Group}.
Suppose that $G$ and $\Gamma$ act on a common standard probability space $Z$ by $\alpha$ and $\beta$, respectively,
in free e.m.p.\ ways.
If the two group actions $\alpha$ and $\beta$ give the same equivalence relation $\R$ on $Z$,
and if $\alpha$ and $\beta$ are separately ergodic,
then there exists a measure preserving map $\theta$ on $Z$ so that its graph is essentially included in $\R$
and that it gives conjugacy between $\alpha$ and $\beta$.
In particular, the outer automorphism group $\mathrm{Out}(\R_\alpha)$ is $\{1\}$.
\end{corollary}

\begin{proof}
We regard $\R$ as an ME coupling
between $G$ and $\Gamma$ with coupling index $1$,
letting $G$ act on the first entry and $\Gamma$ act on
the second entry.
Let $X$ be a $\Gamma$ fundamental domain and $Y$ be a $G$ fundamental domain. Although
the subset $X$ and $Y$
can be identical (for example, the diagonal set), we distinguish them.
We may assume that $H_i \preceq_{\R} \Lambda_i$.
The product $e$ of minimal projections
$e_i \in (L^\infty \R)^{H_i \times \Lambda_i}$ satisfies
$\E_X(e) = 1_X$,
since the groups $\Gamma_{i, \mathrm{fin}}$ in the proof of Theorem
\ref{Theorem; Factorization by Sep Erg Coupling} are $\{1\}$.
By replacing the roles on
$G$ and $\Gamma$, we also get $\E_Y(e) = 1_Y$.
Then there exists a measure preserving map $\theta$ on $Z$ such that
$\chi(\{ (y, \theta(y)) \ | \ y \in Z\}) = e$.

The group homomorphism
$\phi : G \rightarrow \Gamma$ given in the proof of Theorem
\ref{Theorem; OE Str Rigidity for Product Group} is bijective, since
$G_{i,\mathrm{fin}} = \{1\}$, $\prod_{i = 1}^n [ \Gamma_i \colon \Gamma_{i, \mathrm{nor}}] = 1$.
For $g \in G$, we get
\begin{eqnarray*}
            g^{-1} e
        &=& \chi(\{ \alpha(g^{-1})(y), \theta(y)) \ | \ y \in Z\})
         =  \chi(\{ y, \theta(\alpha(g)(y)) \ | \ y \in Z\}),\\
            \phi(g) e
        &=& \chi(\{ (y, \beta(\phi(g)) \theta(y)) \ | \ y \in Z\}).
\end{eqnarray*}
Since $g^{-1} e = \phi(g) e$, there exists a co-mull subset $Z^\prime \subset Z$
such that
\begin{eqnarray*}
        \theta(\alpha(g)(y))
        = \beta(\phi(g))\theta(y), \quad y \in Z^\prime, g \in G,
\end{eqnarray*}
\end{proof}

\begin{theorem}\label{Theorem; OE Str Rigidity for Product Group with Ame Core}
Let $G_0$ $($resp.\ $\Gamma_0$$)$ be an amenable group and let
$G_i\ (1 \le i \le n)$ $($resp.\ $\Gamma_i$$)$ be non-amenable groups in $\mathcal{S}$ with no finite normal subgroup.
Denote $G = G_0 \times \prod_{i = 1} ^ n G_i$
$($resp.\ $\Gamma = \Gamma_0 \times \prod_{i = 1} ^ n \Gamma_i$$)$.
Suppose that $\alpha$ $($resp.\ $\beta$$)$ is a free m.p.\ $G$-action $($resp.\ $\Gamma$-action$)$
on a standard probability space $X$ $($resp.\ $Y$$)$ on which $G_0$
acts $($resp.\ $\Gamma_0$$)$ ergodically.
If the two group actions $\alpha$ and $\beta$ are orbit equivalent,
then there exist $\sigma \in \mathfrak{S}_n$,
group isomorphisms $\phi_i : G_{\sigma(i)} \rightarrow \Gamma_i$ and measure preserving map
$\theta : X \rightarrow Y$ which satisfy:

Define $\phi$ by $\phi \colon \prod_{i=1}^n G_i \ni (g_i)_{\sigma(i)} \mapsto (\phi_i(g_i))_i \in \prod_{i=1}^n \Gamma_i$.
For almost every $x \in X$ and every $g \in \prod_{i = 1} ^ n G_i$,
$\theta(\alpha(g G_0) x) = \beta(\phi(g) \Gamma_0) \theta(x)$.
\end{theorem}

\begin{proof}
We may assume that both of $\alpha$ and $\beta$ are actions on a standard probability space
$Z$ and that they give the same equivalence relation $\Sigma$. We regard $\Sigma$ as an ME coupling
between $G$ and $\Gamma$ with coupling index $1$, letting $G$ act on the first entry and $\Gamma$ act on
the second entry. We choose a $G$ fundamental domain $X$ and $\Gamma$ fundamental domain $Y$.
Define a bijection $\sigma$ by
$H_{\sigma(i)} \preceq_{\Sigma} \Lambda_i$ and $\Lambda_i \preceq_{\Sigma} H_{\sigma(i)}$.
For simplicity, we assume that $\sigma = \mathrm{id}$.

By the previous subsection, $(L^\infty \Sigma)^{H_i \times \Lambda_i}$ is atomic and the
$\Gamma_i$-action on the set of minimal projections is transitive. Since the assumptions are
symmetric on $G$ and $\Gamma$, the $G_i$-action is also transitive. It follows that
for a minimal projection $e_i \in (L^\infty \Sigma)^{H_i \times \Lambda_i}$, the stabilizers
$\Gamma_{i, \mathrm{fin}} \subset \Gamma_i$ and $G_{i, \mathrm{fin}} \subset G_i$ are finite normal subgroups. Thus they are
$\{1\}$. Let $\phi_i : G_i \rightarrow \Gamma_i$ be the group isomorphism given by $g^{-1} e_i = \phi_i(g) e_i$.
The product of projections $e = \prod_{i = 1}^n e_i$ satisfies
$\E^{G_0}_X (e) = \prod_{i = 1}^n |G_{i, \mathrm{fin}}| 1_X = 1_X$.
By replacing the roles on $G$ and $\Gamma$, we also get
$\E^{\Gamma_0}_Y(e) = 1_Y$.

We claim that there exists a measure preserving map $\theta$ on $X$ whose graph
is included in the support of $e$.
Let $e_0$ be maximal among projections dominated by $e$ with the properties
$\E_X(e_0) \le 1_X, \E_Y(e_0) \le 1_Y$.
Suppose that $\int_X \E_X(e_0) d \nu = \int_Y \E_Y(e_0) d \nu = \nu(e_0) < 1$.
By replacing $X$ and $Y$, we may assume that $e_0 \le \chi(X) \le e$ and $e_0 \le \chi(Y) \le e$.
There exists a non-null measurable subset $Y_0 \subset Y$ so that $\chi(Y_0)$ is perpendicular with $e_0$ and that the graph of $Y_0$ gives partial isomorphism on $Z$. Since the $G_0$-action on $\Gamma \backslash Z$ is ergodic,
replacing $Y_0$ with a smaller non-null measurable subset, there exists
$g \in G_0$ satisfying
$
\alpha_g(\E_X (Y_0)) \perp \E_X(e_0)$.
Then the projection $e_0 + g_0 \chi(Y_0)$ is dominated by $e$ and satisfies
\begin{eqnarray*}
    \E_X(e_0 + g \chi(Y_0)) &=& \E_X(e_0) + \alpha_g(\E_X(\chi(Y_0))) \le 1_X,\\
    \quad \E_Y(e_0 + g \chi(Y_0)) &=& \E_Y(e_0) + 1 |_{Y_0} \le 1_Y.
\end{eqnarray*}
This contradicts the maximality of $e_0$. Thus we get $\E_X(e_0) = 1_X$ and $\E_Y(e_0) = 1_Y$.
This means that the projection $e_0$ corresponds to a graph of a measure preserving map
$\theta : Z \rightarrow Z$, that is, $\chi(\{ (x, \theta(x)) \ | \ x \in Z\}) = e_0$.
Then for $g \in \prod_{i = 1}^n G_i$, we have the following equality of projections:
\begin{eqnarray*}
    g^{-1} e
    &=& \sum_{g_0 \in G_0} g^{-1}  g_0^{-1} e_0
     = \chi(\{ (\alpha(g_0^{-1} g^{-1})(x), \theta(x)) \ | \ x \in Z, g_0 \in G_0 \})\\
    &=&  \chi(\{ (x, \theta \alpha(g g_0)(x)) \ | \ x \in Z, g_0 \in G_0 \}),\\
    \phi(g) e
    &=& \sum_{\gamma_0 \in G_0} \phi(g) \gamma_0 e_0
    = \chi(\{ (x, \beta(\phi(g) \gamma_0) \theta(x)) \ | \ x \in Z, \gamma_0 \in \Gamma_0\}).
\end{eqnarray*}
Since $g^{-1} e = \phi(g) e$, it follows that
$
\theta(\alpha(g G_0) x) = \beta(\phi(g) \Gamma_0) \theta(x),\ {\rm a.e.\ } x \in Z$.
\end{proof}

\subsection{OE Super Rigidity Type Theorems}
\label{Subsection; OE Super Rigidity Type Theorems}
\begin{theorem}\label{Theorem; ME SR For Direct Product}
Let $\Gamma = \prod_{i = 1}^n \Gamma_i$ be a direct product group of non-amenable ICC groups in $\mathcal{S}$ and let $G$ be an arbitrary countable group.
\begin{enumerate}
\item
Suppose that there exists an ME coupling $\Sigma$ of $G$ with $\Gamma$.
If the $\Gamma$-action on $G \backslash \Sigma$ is separately ergodic and if
the $G$-action on $\Gamma \backslash \Sigma$ is mildly mixing, then there exists a group homomorphism $\phi \colon G \rightarrow \Gamma$ with finite kernel and the coupling index satisfies $[\Gamma : \phi(G)] = |\ker(\phi)|[\Gamma : G]_\Sigma$.
\item
Suppose that there exist a free separately ergodic m.p.~$\Gamma$-action on a standard probability space
$X$ and a free mildly mixing m.p.~$G$-action on a standard finite measure space $Y$.
If the actions $\alpha$ and $\beta$ are SOE with finite constant, then there exists a homomorphism $\phi \colon G \rightarrow \Gamma$ with finite kernel and finite index image such that
the induced action $\mathrm{Ind}_G^\Gamma(\alpha, \phi)$ is conjugate to $\beta$.
\end{enumerate}
\end{theorem}

The technique we need here has already given by Monod--Shalom. The above theorems are obtained by verbatim translations of the sixth chapter of Monod and Shalom's paper \cite{Monod--Shalom}. We remark that we use the ICC condition on $\Gamma_i$ to construct Furman's homomorphism.

\section{Measure Equivalence between Wreath Product Groups}
\label{Section; Wreath Products}
The goal of this section is Theorem \ref{Introduction; ME between Wreath Product Groups}.

\begin{lemma}\label{Lemma; Wreath Product--Uniqueness of Partial Embedding}
Let $H \subset G$ be an infinite subgroup of a countable group and let $\widetilde{\Gamma} = B \wr \Gamma$ be a
countable wreath product group with $B \neq \{1\}$. Suppose that $\Sigma$ is a measurable embedding
of $G$ into $B \wr \Gamma$.

If $H$ measurably embeds into $\Gamma$ in $\Sigma$, then
there exists a partial embedding $\Omega$ of $H$ into $\Gamma$ such that for any partial embedding $\Omega^\prime$
of $H \preceq_\Sigma \Gamma$, we get $\Omega^\prime \subset \Omega$,
after subtracting a null set.
The $\Gamma$-support of $H \preceq_\Sigma \Gamma$ $($Definition \ref{Definition; Supports}$)$ satisfies
$\E^\Gamma_X (\Omega) = \mathrm{supp}^\Gamma_X(H \preceq_\Sigma \Gamma) \in L^\infty X$.
\end{lemma}

\begin{proof}
We denote $\widetilde{B} = \bigoplus_{\Gamma} B$ and
$p = \mathrm{supp}^\Gamma_X(H \preceq_\Sigma \Gamma)$.
Let $\Omega \subset \Sigma$ be an arbitrary partial embedding of $H$ into $\Gamma$ and let $X$ be a fundamental domain of $\Sigma$ under the $\widetilde{\Gamma}$-action.
We can write $\Omega$ as
$\Omega = \bigsqcup_{b \in \widetilde{B}} \Gamma b X_b$, for some measurable subsets $X_b \subset X$.
The measurable function
$\E^\Gamma_X (\Omega)$ is written as $\sum_{b \in \widetilde{B}} \chi(X_b)$
and it is integrable.
First we claim that $\E^\Gamma_X (\Omega)$ is a projection.

Suppose that the essential range of
$\E^\Gamma_X(\Omega)$ is not contained in $\{0, 1\}$.
Then there exist a non-null measurable subset $W \subset X$
and finite subset $\{b_1, b_2, \cdots, b_k\} \subset \widetilde{B}$
satisfying $k \ge 2, b_i \neq b_j\ (i \neq j)$ and
$\Omega \cap \Gamma W = \bigsqcup_{i = 1}^k \Gamma b_i W$.
The measurable set $b_1^{-1} \Omega \cap b_2^{-1} \Omega$
is $H$-invariant and satisfies
\begin{eqnarray*}
    b_1^{-1} \Omega \cap b_2^{-1} \Omega \cap \Gamma W
    = \bigcup_i b_1^{-1} \Gamma b_i W \cap \bigcup_j b_2^{-1} \Gamma b_j W
    = \bigcup_{i,j} (b_1^{-1} \Gamma b_i \cap b_2^{-1} \Gamma b_j)W.
\end{eqnarray*}
Applying the function valued measure $\E_X \colon \widehat{L^\infty(\Sigma)_+} \rightarrow \widehat{L^\infty(X)_+}$, we get
\begin{eqnarray*}
    \E_X(b_1^{-1} \Omega \cap b_2^{-1} \Omega) 1_W
    = \left| \bigcup_{i,j} (b_1^{-1} \Gamma b_i \cap b_2^{-1} \Gamma b_j) \right| 1_W.
\end{eqnarray*}
Since $\bigcup_{i,j}b_1^{-1} \Gamma b_i \cap b_2^{-1} \Gamma b_j$ is a finite set and non-empty,
we get $H \preceq_{\Sigma} \{1\}$ (Lemma \ref{Lemma; FVM and MEm}). This contradicts $|H| = \infty$.
Thus the essential range of $\E^\Gamma_X(\Omega)$ is included in $\{0,1\}$ and $\E^\Gamma_X(\Omega)$ is a projection.

When $\Omega, \Omega^\prime$ are partial embeddings of $H$ into $\Gamma$,
the union
$\Omega \cup \Omega^\prime$ is also a partial embedding of $H$ into $\Gamma$.
By the above,
$\E^\Gamma_X(\Omega \cup \Omega^\prime)$ is a projection.

There exists an increasing sequence of $\Omega_n$ of partial embeddings of $H$ into $\Gamma$ with
$\bigvee_n \E^\Gamma_X(\Omega_n) = p$.
Let $\Omega$ be the union of $\{\Omega_n\}$. Applying $\E^\Gamma_X$, we get
\begin{eqnarray*}
    \E^\Gamma_X ( \chi (\Omega) )
    = \sup_n \E^\Gamma_X (\Omega_n) = p.
\end{eqnarray*}
It follows that $\Omega$ is again a partial embedding of $H$ into $\Gamma$.
Let $\Omega^\prime$ be another partial embedding. Then we get
$
    \E^\Gamma_X(\Omega)
    \le \E^\Gamma_X(\Omega \cup \Omega^\prime)
    \le p
    = \E^\Gamma_X(\Omega).
$
Since $\E^\Gamma_X$ is faithful, we conclude $\chi(\Omega \cup \Omega^\prime) = \chi(\Omega)$
and that $\Omega$
dominates all partial embedding, after subtracting a null set.
\end{proof}

\begin{proposition}
\label{Proposition; Wreath Product--Uniqueness of Partial Embedding}
Let $G \times H \subset \widetilde{G}$ be a direct product type subgroup of an exact group $\widetilde{G}$.
Let $\widetilde\Gamma$ be an exact wreath product group $B \wr \Gamma$ with amenable base
$B \neq \{1\}$. Suppose that $G$ is non-amenable and that $H$ is infinite.

If $\Sigma$ is an ergodic measurable embedding of $\widetilde{G}$ into $\widetilde\Gamma$, then there exists
a maximal partial embedding $\Omega$ of $G \times H$ into $\Gamma$.
The embedding satisfies
$\E^\Gamma_X (\Omega) = 1_X \in L^\infty X$.
\end{proposition}

\begin{proof}
The group $\widetilde\Gamma$ is bi-exact relative to $\{\Gamma \}$ by Lemma \ref{Lemma; Wreath Product}.
By Theorem \ref{Theorem; Non embeddability implies amenability},
$H$ measurably embeds into $\Gamma$ in $\Sigma$. Furthermore, its
$\widetilde\Gamma$-support of the embedding is $1_X$.
Let $\Omega$ be the largest embedding of $H$ into $\Gamma$
(Lemma \ref{Lemma; Wreath Product--Uniqueness of Partial Embedding}).

Since $g \in G$ commutes with all elements in $H$,
the measurable subsets $g \Omega, g^{-1} \Omega$ also
give embeddings of $H$ into $\Gamma$. The maximality of $\Omega$ means that
$g \Omega \subset \Omega$ and $g^{-1} \Omega \subset \Omega$,
after null sets are subtracted.
It follows that the difference between $g \Omega$ and $\Omega$ is null.
We may assume that $\Omega$ is $G \times H$-invariant. The measurable
subset $\Omega$ gives a measurable embedding of $G \times H$ into $\Gamma$.
The embedding $\Omega$ of $G \times H$ into $\Gamma$ is maximal, since
it is maximal as an embedding of $H$.
\end{proof}

\begin{proof}
[Proof for Theorem \ref{Introduction; ME between Wreath Product Groups}]
Let $\Sigma$ be an ergodic ME coupling between two wreath products
$\widetilde{G}$ and $\widetilde{\Gamma}$.
By Proposition \ref{Proposition; Wreath Product--Uniqueness of Partial Embedding}, we take
the largest embedding $\Omega_l \subset \Sigma$ of $G \times H$ into $\Gamma \times \Lambda$
and the largest embedding $\Omega_r \subset \Sigma$ of $\Gamma \times \Lambda$ into $G \times H$.
It suffices to show that the difference between $\Omega_l$, $\Omega_r$ is null. Since the assumptions are symmetric,
we only prove that $\Omega_l \cap \Omega_r^c$ is null.

By the equality
$\E^{G \times H}_Y (\Omega_r) = 1_Y$,
there exists a measurable subset $Y^\prime \subset \Omega_r$ so that $Y^\prime$ is a fundamental domain for the $\widetilde{G}$-action on $\Sigma$ and that
$\chi((G \times H) Y^\prime) = \chi(\Omega_r)$.
Denote $\widetilde{A} = \oplus_{G \times H} A$.
We may assume that $\Omega_r$ is an $\widetilde{A}$-fundamental domain for the action  $\widetilde{A} \curvearrowright \Sigma$.

Suppose that $\Omega_l \cap (\Omega_r)^c$ is not null.
Then there exists $1 \neq a \in \widetilde{A}$
such that $\Omega_l \cap a \Omega_r$ is not null.
We note that
this is $\Gamma \times \Lambda$-invariant.
There exist infinitely many elements $\{g_i\}_{i \in I}$ in $G \times H$
such that $\{g_i(a)\}_{i \in I}$ are different from each other.
The following equation holds true
\begin{eqnarray*}
        \Tr_{\Gamma \times \Lambda}(\Omega_l \cap g_i(a) \Omega_r)
     =  \Tr_{\Gamma \times \Lambda}(g_i (\Omega_l \cap a \Omega_r))
     =  \Tr_{\Gamma \times \Lambda}(\Omega_l \cap a \Omega_r).
\end{eqnarray*}
Since the measurable subsets
$\{ g_i(a) \Omega_r\}$ are disjoint, we get
\begin{eqnarray*}
    0
    < \sum_{i \in I}
            \Tr_{\Gamma \times \Lambda}(\Omega_l \cap g_i(a) \Omega_r)
    \le \Tr_{\Gamma \times \Lambda}(\Omega_l) < \infty.
\end{eqnarray*}
This contradicts $|I| = \infty$. We conclude that $\Omega_l \subset \Omega_r$,
after subtracting a null set.
Since the assumptions are symmetric, we get $\Omega_l = \Omega_r$,
after subtracting null sets. The measurable subset $\Omega_l = \Omega_r$
gives an ME coupling of $G \times H$ with $\Gamma \times \Lambda$.

For the second assertion, we suppose that the coupling $\Sigma$
comes from SOE, in other words,
the dot actions $\alpha \colon \widetilde{G} \curvearrowright X$ and
$\beta \colon \widetilde\Gamma \curvearrowright Y$ are free.
We further assume that the actions $\alpha |_{G \times H}$, $\beta |_{\Gamma \times \Lambda}$ are ergodic.
Since $\E^{\Gamma \times \Lambda}_X (\Omega_l) = 1_X$,
the action $\alpha |_{G \times H}$ is conjugate to the dot action
$G \times H \curvearrowright (\Gamma \times \Lambda) \backslash
\Omega_l$. By symmetricity, the action
$\beta |_{\Gamma \times \Lambda}$ is conjugate to the dot action
$\Gamma \times \Lambda \curvearrowright (G \times H) \backslash
\Omega_l$. Choose an embedding from $X$ to a $\Gamma \times \Lambda$ fundamental domain of $\Omega_l$ and an embedding from $Y$ to a $G \times H$ fundamental domain of $\Omega_l$. The compositions $p \colon X \hookrightarrow \Omega_l \rightarrow (G \times H) \backslash \Omega_l \cong Y$ and $q \colon Y \hookrightarrow \Omega_l \rightarrow (\Gamma \times \Lambda) \backslash \Omega_l \cong X$ gives SOE (weak OE) between $\alpha |_{G \times H}$ and $\beta |_{\Gamma \times \Lambda}$.
\end{proof}

\section{Factorization of Amalgamated Free Products}
\label{Section; Amalgamated Free Products}

The goal of this section is Theorem \ref{Theorem; B--S rigidity}. We start with an argument on Bass--Serre trees.

\begin{lemma}\label{Lemma; BS Tree}
Let $\Gamma$ be an amalgamated free product $\Gamma_1 \ast_B \Gamma_2$
of countable groups. Let $i$ be either $1$ or $2$ and $u$ be an element of $\Gamma$.
If $u \Gamma_i \neq \Gamma_1$,
then there exist $\gamma \in \Gamma$
and a subgroup $B_u \subset \gamma B \gamma^{-1}$ with the following property:\\
For all $s, t \in \Gamma$, $S = s \Gamma_i \cap t \Gamma_1 u \subset \Gamma$ is either empty
or a left coset of $B_u$.
\end{lemma}

\begin{proof}
Fix $u$ throughout of this proof. Let $s, t$ be arbitrary elements in $\Gamma$.
Let $T = \Gamma / \Gamma_1 \sqcup \Gamma / \Gamma_2$
be the Bass--Serre tree for
$\Gamma = \Gamma_1 \ast_B \Gamma_2$, on which the group $\Gamma$ acts.
The set $t \Gamma_1 u$ is identical to the collection of elements which move
$u^{-1} \Gamma_1 \in T$ to $t \Gamma_1 \in T$.
The set $s \Gamma_i$ is the collection of elements which move
$\Gamma_i \in T$ to $u \Gamma_i \in T$.

Let $B_u$ be the set of elements which stabilize all points
\begin{eqnarray*}
\{u^{-1} \Gamma_1 = p_1, p_2, \cdots, p_l = \Gamma_i\} \subset T
\end{eqnarray*}
on the geodesic
from $u^{-1} \Gamma_1$ to $\Gamma_i$.
Suppose that $S$ is not empty. We take an element
$v \in S$. Then the set $S$ is of the form $v B_u$. Any element $b \in B_u$ stabilizes
the edge $\{ p_{l - 1}, p_l = \Gamma_i \}$.
The stabilizer of $\{ p_{l - 1}, p_l = \Gamma_i \}$ is
of the form $\gamma B \gamma^{-1}$ for some $\gamma \in \Gamma$.
It follows that $B_u$ is a subgroup of $\gamma B \gamma^{-1}$.
\end{proof}

\begin{lemma}\label{Lemma; AFProduct--Uniqueness of Partial Embedding}
Let $H \subset G$ be an inclusion of countable groups and
let $\Gamma = \Gamma_0 \ast_B \Lambda$
be a free product with amalgamation over a common subgroup $B$.
Suppose that $\Sigma$ is a measurable embedding
of $G$ into $\Gamma$.

If $H \preceq_\Sigma \Gamma_0$
and if $H \not\preceq_\mathrm{ME} B$, then
there exists a partial embedding $\Omega$ of $H$ into $\Gamma_0$ in $\Sigma$,
which is maximal.
Namely, for any partial embedding $\Omega^\prime$ of $H$ into $\Gamma_0$,
$\Omega^c \cap \Omega^\prime$ is a null set. Furthermore, the $\Gamma$-support of $H \preceq_\Sigma \Gamma_0$ satisfies
$\mathrm{supp}^{\Gamma}_X (H \preceq_\Sigma \Gamma_0)
= \E^{\Gamma_0}_X (\Omega)$.
\end{lemma}

\begin{proof}
We choose and fix representatives $\{s_\iota\}_{\iota \in I}$ of the
right cosets $\Gamma_0 \backslash \Gamma$.
Let $\Omega \subset \Sigma$ be an arbitrary partial embedding of $H$ into $\Gamma_0$ and let $X$ be a fundamental domain of $\Sigma$ under the $\Gamma$-action.
We can write $\Omega = \bigsqcup_{\iota \in I} \Gamma_0 s_\iota X_\iota$,
for some measurable subsets $X_\iota \subset X$. The measurable function
$\E^{\Gamma_0}_X(\Omega) = \sum_{\iota \in I} \chi(X_\iota)$ is integrable.

Suppose that the essential range of $\E^{\Gamma_0}_X(\Omega)$
is not included in $\{0, 1\}$.
Then there exist a non-null measurable subset $W \subset X$
and a finite subset $\{s_1, s_2, \cdots, s_k\} \subset \{s_\iota\}_{\iota \in I}$
satisfying $k \ge 2, s_i \neq s_j\ (i \neq j)$ and
$\Omega \cap \Gamma W = \bigcup_{i = 1}^k \Gamma_0 s_i W$.
Replacing $X$ with $s_1 W \sqcup (X \cap (W)^c)$ and $\{s_i\}$ with $\{s_i s_1^{-1}\}$,
we may assume $s_1 = 1$.

The measurable set $s_2 \Omega \cap \Omega$
is $H$-invariant and satisfies
\begin{eqnarray*}
    s_2 \Omega \cap \Omega = \bigcup_i s_2 \Gamma_0 s_i W \cap \bigcup_j \Gamma_0 s_j W
    = \bigcup_{i,j} (s_2 \Gamma_0 s_i \cap \Gamma_0 s_j)W.
\end{eqnarray*}
By Lemma \ref{Lemma; BS Tree}, there exists a subgroup $B_2 \subset \gamma^{-1} B \gamma$ for some $\gamma \in \Gamma$ so that
$S = \bigcup_{i,j} (s_2 \Gamma_0 s_i \cap \Gamma_0 s_j)$
is a finite union of right cosets of $B_2$.
The set $S$ is not empty since $s_2$ is an element of $S$.
The function valued measure of $b_1 \Omega \cap b_2 \Omega$ with respect to $B_2$ satisfies
\begin{eqnarray*}
    \E^{B_2}_X(s_2 \Omega \cap \Omega) |_{W}
    = \left| B_2 \backslash S \right| 1_{W}.
\end{eqnarray*}
Thus we get $H \preceq_{\Sigma} B_2 \subset \gamma^{-1} B \gamma$ (Lemma \ref{Lemma; FVM and MEm}). Since $H \not\preceq_\mathrm{ME} B$,
this is a contradiction.
We conclude that the essential range of
$\E^{\Gamma_0}_X(\Omega)$ is included in $\{0, 1\}$.
For the rest of the proof,
we do the same argument as Lemma
\ref{Lemma; Wreath Product--Uniqueness of Partial Embedding}.
\end{proof}

\begin{proposition}\label{Proposition; AFP--Uniqueness of Partial Embedding}
Let $G_1 \times H \subset G$ be a direct product type subgroup of an exact group $G$.
Let $\Gamma$ be an exact free product group $\ast_B \Gamma_i\  (1 \le i \le n)$ with amalgamation
over a common amenable subgroup
$B$. Suppose that $G_1, H$ are non-amenable.
If $\Sigma$ is an ergodic measurable embedding of $G$ into $\Gamma$, then
\begin{enumerate}
\item
The $\Gamma$-supports $p_i$ of $G_1 \times H \preceq_\Sigma \Gamma_i$
are mutually orthogonal and satisfy
$\sum_{i = 1}^n p_i= 1_X$.
\item
There exist maximal measurable embeddings
$\Omega_i \subset \Sigma$ of $G_1 \times H$ into $\Gamma_i$.
Their function valued measure $\E^{\Gamma_i}_X = \E^{(i)}_X$ satisfies
$\E^{(i)}_X (\Omega_i) = p_i$.
\end{enumerate}
\end{proposition}

\begin{proof}
Since $\Gamma$ is bi-exact relative to $\{\Gamma_i\}$ and $G_1$ is non-amenable,
$H$ measurably embeds into some $\Gamma_i$ in $\Sigma$ (Theorem \ref{Theorem; Non embeddability implies amenability}).
Define $p_i$ as the
$\Gamma$-support of the embedding $H \preceq_\Sigma \Gamma_i$ instead of $G_1 \times H \preceq_\Sigma \Gamma_i$.
By the maximality argument,
Theorem \ref{Theorem; Non embeddability implies amenability}
tells that the union of the $\Gamma$-supports covers $X$, that is,
$\bigvee_i p_i = 1_X$.
Since non-amenable group $H$ does not measurably embed into amenable group $B$,
we can take the largest partial embedding $\Omega_i$ of $H$ into $\Gamma_i$
(Lemma \ref{Lemma; AFProduct--Uniqueness of Partial Embedding}).
The function valued measure satisfies $\E_X^{(i)}(\Omega_i) = \supp_X^\Gamma(H \preceq_\Sigma \Gamma_i)$.

Since $g \in G_1$ commutes with all elements in $H$,
the measurable subsets $g \Omega, g^{-1} \Omega$ also
give embeddings of $H$ into $\Gamma$. By the maximality of $\Omega$, we have
$g \Omega \subset \Omega$ and $g^{-1} \Omega \subset \Omega$,
after subtracting null sets.
We may assume that $\Omega$ is $G_1 \times H$-invariant. The measurable
subset $\Omega_i$ gives a measurable embedding of $G_1 \times H$ into $\Gamma_i$.
The maximal embedding $\Omega$ of $H$  into
$\Gamma_i$ is also maximal as an embedding of $G_1 \times H$.
The support of $G_1 \times H \preceq_\Sigma \Gamma_i$ satisfies
\begin{eqnarray*}
p_i = \E_X^{(i)}(\Omega_i) \le \supp_X^\Gamma(G_1 \times H \preceq_\Sigma \Gamma_i) \le \supp_X^\Gamma(H \preceq_\Sigma \Gamma_i).
\end{eqnarray*}
It follows that
$p_i = \supp_X^\Gamma(G_1 \times H \preceq_\Sigma \Gamma_i)$.

We claim that the projections $p_i$ are mutually orthogonal. It suffices to show that
the $\Gamma$-support $P_i$ for the embedding $G_1 \times H \preceq_\Sigma \ast_{B, j \neq i}\Gamma_j$
is perpendicular to $p_i$.
Denote $\Lambda = \ast_{B, j \neq i}\Gamma_j$.
Suppose that $P_i p_i \neq 0$. Then there exists a partial embedding
$\Omega^\prime \subset \Sigma$
of $G_1 \times H$ into $\Lambda$ such that
$\E^\Lambda_X(\Omega^\prime) p_i \neq 0$. Since $\E^\Lambda_X(\Omega^\prime)$ is a projection, there exist a non-null measurable subset $W \subset X$ and $s, t \in \Gamma$
such that
\begin{eqnarray*}
    \Omega \cap \Gamma W = \Gamma_i s W, \quad
    \Omega^\prime \cap \Gamma W = \Lambda t W.
\end{eqnarray*}
By Lemma \ref{Lemma; BS Tree}, the set $t s^{-1} \Gamma_i s \cap \Lambda t$ is a right coset
of a subgroup $C = C_{t s^{-1}} \subset \Gamma$ which is isomorphic to a subgroup of $B$. The function valued measure
$\E^{C}_X$ of $t s^{-1}\Omega \cap \Omega^\prime$ satisfies
\begin{eqnarray*}
        \E^{C}_X(t s^{-1}\Omega \cap \Omega^\prime) 1_W
    = \E^{C}_X((t s^{-1}\Gamma_i s \cap \Lambda t) W)
    = 1_W.
\end{eqnarray*}
This means that $G_1 \times H \preceq_\Sigma C$, which contradicts non-amenability of
$G_1 \times H$.
It follows that $p_i$ is perpendicular to $P_i$ and that $\{p_i\}$ are mutually orthogonal.
\end{proof}

\begin{theorem}\label{Theorem; B--S rigidity}
Let $G_i\ (1 \le i \le m)$ and $\Gamma_j\ (1 \le j \le n)$ be
direct products of two non-amenable exact groups. Suppose that $\{G_i\}$ have
a common amenable subgroup $A$ and that $\{\Gamma_j\}$ have
a common amenable subgroup $B$.
Denote by $G$, $\Gamma$ the amalgamated free products
$G = \ast_A G_i$, $\Gamma = \ast_B \Gamma_j$.
Then we have the following:
\begin{enumerate}
\item
If $G \sim_\mathrm{ME} \Gamma$, then for any $1 \le i \le m$ there
exists $1 \le \sigma(i) \le n$ satisfying
$G_i \sim_\mathrm{ME} \Gamma_{\sigma(i)}$
and for any $1 \le j \le n$ there exists
$1 \le \rho(j) \le m$ satisfying
$G_{\rho(j)} \sim_\mathrm{ME} \Gamma_j$;

\item
If $m = n = 2$ and $G \sim_\mathrm{ME} \Gamma$, then there exists $i \in \{1,2\}$ satisfying
$G_1 \sim_\mathrm{ME} \Gamma_i$, $G_2 \sim_\mathrm{ME} \Gamma_{i + 1}$,
where $i+1 \in \{1,2\} \cap \{i\}^c$;

\item
Let $\Sigma$ be an ME coupling between $G$ and $\Gamma$.
If the $G_i \times \Gamma$-action
on $\Sigma$ is ergodic for any $i$
and if $G \times \Gamma_j$-action on $\Sigma$ is ergodic
for any $j$,
then $m=n$ and there exists $\sigma \in \mathfrak{S}_n$ satisfying
$G_i \sim_\mathrm{ME} \Gamma_{\sigma(i)}$.
More precisely, there exist $G_i \times \Gamma_{\sigma(i)}$-invariant measurable subsets $\Omega(i, \sigma(i)) \subset \Sigma$ which gives an ME coupling of $G_i$ with $\Gamma_{\sigma(i)}$ and satisfies
$[\Gamma : G]_{\Sigma} = [\Gamma_{\sigma(i)}: G_i]_{\Omega(i, \sigma(i))}$;

\item
Let $\alpha$ be a free m.p.\ $G$-action on standard probability space $X$ and let $\beta$ be a free m.p.\ $\Gamma$-action on a standard finite measure space $Y$.
Suppose that the $G_i$-action $\alpha |_{G_i}$ on $X$ and the
$\Gamma_j$-action $\beta |_{\Gamma_j}$ on $Y$ are ergodic
for any $i, j$. If the $G$-action and $\Gamma$-action are
SOE, then $m=n$ and there exists $\sigma \in \mathfrak{S}_n$ so that
$\alpha |_{G_i}$ and $\beta |_{\Gamma_{\sigma(i)}}$ are SOE.
\end{enumerate}
\end{theorem}

\begin{proof}
Let $\Sigma$ be an ergodic ME coupling between two amalgamated
free products $G$ and $\Gamma$
and let $X, Y$ be fundamental domains for the $\Gamma$-action and $G$-action, respectively.
We write $G = G_i \ast_A H_i, \Gamma = \Gamma_j \ast_B \Lambda_j$.

Denote by $\Omega(i,j) \subset \Sigma$ the
(possibly null)
maximal partial embedding of $G_i$ into $\Gamma_j$ in $\Sigma$
in Proposition \ref{Proposition; AFP--Uniqueness of Partial Embedding}.
The functions $\{ \E^{\Gamma_j}_X(\Omega(i,j)) \}$ are
characteristic functions and satisfy
\begin{eqnarray*}
    \sum_{j =1}^n \E^{\Gamma_j}_X(\Omega(i,j)) = 1_X,
     \quad i = 1, 2, \cdots, m.
\end{eqnarray*}
Since the assumptions are symmetric,
again by Proposition \ref{Proposition; AFP--Uniqueness of Partial Embedding},
we get
the maximal partial embeddings $\Xi(i, j)$
of $\Gamma_j \preceq_\Sigma G_i$.
The functions
$\E^{G_i}_Y(\Xi(i,j))$
are characteristic functions and satisfy
\begin{eqnarray*}
    \sum_{i =1}^m \E^{G_i}_Y(\Xi(i,j)) = 1_X,
     \quad j = 1, 2, \cdots, n.
\end{eqnarray*}

First we claim that
$\supp^G_Y(\Omega(i,j)) \le \supp^G_Y(\Xi(i,j))$.
We have only to show that if $1 \le i, k \le m$ satisfy
$\E^{G_i}_Y(\Omega(i,j)) \E^{G_k}_Y(\Xi(k,j)) \neq 0$, then $i = k$.
Under the assumption, there exists $h \in G$ such that
$\Omega(i,j) \cap h (\Xi(k,j))$ is non-null.
Since the essential range of $\E^{G_k}_Y(\Xi(k,j))$ is contained by
$\{0, 1\}$, there exists a measurable subset $Y_k \subset \Sigma$ such that
\begin{eqnarray*}
    \Xi(k,j) =  G_k Y_k, \quad
    h \Xi(k,j) = h G_k Y_k,
\end{eqnarray*}
after subtracting null sets.
Suppose $k \neq i$.
For $g \in G_i \cap A^c$, the $\Gamma_i$-invariant measurable subsets
$h \Xi(k,j)$ and $g h \Xi(k,j)$ are almost disjoint.
Letting $\{g_\iota\}_{\iota \in I}$ be representatives for the left cosets
$G_i / A$,
we get that $\{g_\iota h \Xi(k,j)\}_{\iota \in I}$ are almost disjoint and
\begin{eqnarray*}
    0
    < \Tr_{\Gamma_j} \left(
          \Omega(i,j) \cap \bigsqcup_{\iota \in I} g_\iota h \Xi(k,j) \right)
    \le \Tr_{\Gamma_j} (\Omega(i,j)) < \infty.
\end{eqnarray*}
The measurable subsets $\Omega(i,j) \cap g_\iota h \Xi(k,j)$
equal to $g_\iota (\Omega(i,j) \cap h \Xi(k,j))$ and have the same value
of $\Tr_{\Gamma_j}$.
This contradicts $|I| = [G_i : A]= \infty$.
The first claim was confirmed.

We next claim that $\Omega(i,j)$ is essentially included in $\Xi(i,j)$.
By the last paragraph, we get $\chi(\Omega(i,j)) \le \bigvee_{h \in G} h\chi(\Xi(i,j))$.
it suffices to deduce a contradiction supposing $h \in G \cap G_i^c$ satisfies
$\chi(\Omega(i, j)) h \chi(\Xi(i,j)) \neq 0$.
For $g \in G_i \cap B^c$, the measurable subsets
\begin{eqnarray*}
    h \Xi(i,j) = h G_i Y_i, \quad g h \Xi(i,j) =  g h G_i Y_i
\end{eqnarray*}
are disjoint.
By the same calculation as the last paragraph, we get
\begin{eqnarray*}
    0
    < |I| \Tr_{\Gamma_j} (\Omega(i,j) \cap h \Xi(i,j))
    \le \Tr_{\Gamma_j} (\Omega(i,j)) < \infty.
\end{eqnarray*}
We get a contradiction with
$|I| = [G_i : A]= \infty$.
We conclude that $\chi(\Omega(i,j)) \le \chi(\Xi(i,j))$. Since the
assumptions are symmetric on $G$ and $\Gamma$, it follows that
that $\Omega(i,j) = \Xi(i,j)$ after subtracting null sets.

The measurable set $\Omega(i,j) = \Xi(i,j)$ gives an ME coupling of $G_i$ with $\Gamma_j$ if it is non-null. For every $1 \le i \le m$ there exists $1 \le j \le n$
satisfying $\E^{\Gamma_j}_X(\Omega(i, j)) \neq 0$. This means that $\Omega(i, j)$ is non-null and
$G_i \sim_\mathrm{ME} \Gamma_j$. By the same way,
for $1 \le j \le n$ there exists $1 \le i \le m$
satisfying $G_i \sim_\mathrm{ME} \Gamma_j$.
We get the first assertion.

Suppose $m = n = 2$. By the first assertion, there exist $i, j \in \{1,2\}$ such that
$G_1 \sim_{\mathrm{ME}} \Gamma_i$, $G_2 \sim_{\mathrm{ME}} \Gamma_j$.
If $i = j$, then there exists $k \in \{1,2\}$ satisfying
$G_k \sim_{\mathrm{ME}} \Gamma_{i+1}$
again by the first assertion. Then we get the second assertion.

We next suppose that the $G_i \times \Gamma$-action on $\Sigma$ is ergodic
for any $1 \le i \le  m$
and that the $G \times \Gamma_j$-action on $\Sigma$ is ergodic
for any $1 \le i \le n$.
Since the $G_i$-action on $X \cong \Gamma \backslash \Sigma$ is ergodic,
the function $\E^{\Gamma_j}_X (\Omega(i, j))$ is either $0$ or $1_X$.
It follows that for $1 \le i \le m$ there exists a unique
$1 \le j = \sigma(i) \le n$ such that
$\Omega(i, j)$ is non-null.
Since the assumptions are symmetric,
for $1 \le j \le n$ there exists a unique
$1 \le \rho(j) \le m$ such that
$\Omega(i, j)$ is non-null.
The maps $\sigma$ and $\rho$ must be the inverse maps of each other,
and in particular $m = n$.
Since the measure of a $\Gamma_{\sigma(i)}$ fundamental domain
of $\Omega(i, \sigma(i))$ is
\begin{eqnarray*}
    \Tr_{\Gamma_{\sigma(i)}}(\Omega(i, \sigma(i)))
    = \int_X \E^{{\Gamma_{\sigma(i)}}}_X(\Omega(i, \sigma(i))) d \nu
    = \nu(X),
\end{eqnarray*}
and that of a $G_i$ fundamental domain is
\begin{eqnarray*}
    \Tr_{G_i}(\Omega(i, \sigma(i)))
    = \int_Y \E^{G_i}_Y (\Omega(i, \sigma(i))) d \nu
    = \nu(Y),
\end{eqnarray*}
we get the following equation between two coupling indices,
\begin{eqnarray*}
    [\Gamma : G]_{\Sigma} = \nu(Y) / \nu(X)
    = [\Gamma_{\sigma(i)}: G_i]_{\Omega(i, \sigma(i))}.
\end{eqnarray*}

Suppose that the coupling $\Sigma$ comes from SOE, in other words,
the actions $G \curvearrowright X \cong \Gamma \backslash \Sigma$ and $\Gamma \curvearrowright Y \cong G\backslash \Sigma$ are essentially
free and that the actions $G_i \curvearrowright X$, $\Gamma_j \curvearrowright Y$ is ergodic. Then the actions
$G_i \curvearrowright \Gamma_\sigma(i) \backslash \Omega(i, \sigma(i))$,
$\Gamma_{\sigma(i)} \curvearrowright G_i \backslash \Omega(i, \sigma(i))$ are conjugate to
the original dot actions. It follows that the coupling $\Omega(i, \sigma(i))$ give
the stable orbit equivalence between two actions
$G_i \curvearrowright X$ and $\Gamma_{\sigma(i)} \curvearrowright Y$.
\end{proof}

\begin{acknowledgment}
This paper was written during the author's stay in UCLA.
The author is grateful to Professor Sorin Popa and Professor Narutaka Ozawa
for their encouragement and fruitful conversations.
He is supported by JSPS Research Fellowships for Young Scientists.
\end{acknowledgment}

\end{document}